\documentclass[10pt]{amsart}
\usepackage{amsmath}
\usepackage{amssymb}
\usepackage{amsthm}
\usepackage{amsfonts, dsfont}
\usepackage{paralist}
\usepackage{graphics} 
\usepackage{epsfig} 
\usepackage{graphicx}  
\usepackage{epstopdf}
\usepackage{epstopdf}
\usepackage{verbatim}
\usepackage{mathrsfs}
\usepackage{mathtools}
\usepackage{pstricks}

\usepackage{relsize}
\usepackage{tikz}
\usetikzlibrary{matrix}
\usepackage{subcaption}
\usepackage{pgfplots}
\usepackage{fixltx2e}
\usepackage{enumitem}
\usepackage{ upgreek }
\usepackage[colorlinks=true]{hyperref}
\hypersetup{urlcolor=blue, linkcolor=blue, citecolor=red}
\usepackage{hyperref}
\usepackage{cleveref}
\usepackage{bm}
\usepackage{appendix}
\newtheorem{theorem}{Theorem}[section]
\newtheorem{lemma}[theorem]{Lemma}
\newtheorem{corollary}[theorem]{Corollary}
\newtheorem{proposition}[theorem]{Proposition}
\newtheorem{remark}[theorem]{Remark}
\newtheorem{definition}[theorem]{Definition}

\numberwithin{equation}{section}

\DeclareMathOperator*{\argmin}{argmin}

\begin{document}
\title[Viscous Ergodic Problem. Part 2: Mean-Field Games]{A viscous ergodic problem with unbounded and measurable ingredients. Part 2: Mean-Field Games}
\thanks{The author is funded by the Deutsche Forschungsgemeinschaft (DFG, German Research Foundation) – Projektnummer 320021702/GRK2326 – Energy, Entropy, and Dissipative Dynamics (EDDy). These results were presented during the 19th International Symposium on Dynamic Games and Applications held in Porto (Portugal) in July 25-28, 2022. An earlier, yet incomplete, version of this manuscript was part the author's Ph.D. thesis \cite{kouhkouhPhD} which was conducted when he was a Ph.D. student at the University of Padova.}

\author{Hicham Kouhkouh}
\address{Hicham Kouhkouh \newline \indent
{RWTH Aachen University, Institut f\"ur Mathematik,  \newline \indent 
RTG Energy, Entropy, and Dissipative Dynamics,\newline \indent
Templergraben 55 (111810)},
 \newline \indent 
{52062, Aachen, Germany}
}
\email{\texttt{kouhkouh@eddy.rwth-aachen.de}}

\date{\today}
%
\begin{abstract}
We address the problem of existence and (non-)uniqueness of solutions $\big(c,u(\cdot),\mu\big)$ to ergodic mean-field games in the whole space $\mathds{R}^{m}$ with unbounded and merely measurable data, and for non-separable Hamiltonian. The payoff functional satisfies a new monotonicity condition, different from the usual one due to Lasry and Lions. The method we use is also different from classical approaches. It relies on duality theory and optimization in abstract Banach spaces together with maximal dissipativity of diffusion operators, and it follows the companion paper \cite{kouhkouh1}.
\end{abstract}

\maketitle


\section{Introduction}

This manuscript is devoted to the problem of existence of solutions to ergodic mean-field games (MFG for short) in the whole space $\mathds{R}^{m}$ with unbounded data satisfying subexponential growth. The MFG system is made of two coupled partial differential equations: the first equation is a Hamilton-Jacobi-Bellman equation (HJB for short), the second one is a nonlinear Fokker-Planck-Kolmogorov equation (FPK for short). And the corresponding ergodic problem is the following:
\begin{equation}
    \label{eq: mfg - intro}
    \begin{aligned}
    &\quad\quad\textit{Find } (c,u, \mu)\in\mathds{R}\times \mathcal{X}(\mathds{R}^{m})\times\mathcal{P}(\mathds{R}^{m}),\, \textit{s.t.:}\\ 
            & H(x,\nabla u(x),D^{2}u(x),\mu) = c\quad  \text{ and }\; - \mathcal{L}^{*}_{\upalpha_{[u,\mu]}}\mu = 0 
    \end{aligned}
\end{equation}
where $\upalpha_{[u,\mu]}(\cdot)$ is a function of $x$ and it depends on $u$ and $\mu$ such that
\begin{equation*}
	\upalpha_{[u,\mu]}(x) \in \argmin\limits_{\alpha \in A}\{\,-\mathcal{L}_{\alpha}u(x) + f(x,\alpha,\mu)\,\}.
\end{equation*}
Here $\mathcal{X}$ is a functional space (part of the unknowns), $\mathcal{P}$ is the set of probability measures, the Hamiltonian is of the form \begin{equation*}
     H(x,\nabla u(x),D^{2}u(x),\mu)\coloneqq \min\limits_{\alpha\in A}\{\,-\mathcal{L}_{\alpha}u(x) + f(x,\alpha,\mu)\,\},
\end{equation*}
the diffusion operator $\mathcal{L}_{\alpha}$ is a linear operator given by
\begin{equation*}
    \mathcal{L}_{\alpha}\varphi(x) \coloneqq \text{trace}\big( a(x,\alpha)D^{2}\varphi(x)\big) + b(x,\alpha)\cdot\nabla\varphi(x) 
\end{equation*}
and its adjoint $\mathcal{L}^{*}_{\alpha}$ is then 
\begin{equation*}
    \mathcal{L}^{*}_{\alpha}\mu(x) = \text{trace}\big(D^{2}(a(x,\alpha)\mu(x))\big) - \text{div}\big(b(x,\alpha)\mu(x)\big).
\end{equation*}
The second equation in \eqref{eq: mfg - intro} is  $-\mathcal{L}^{*}_{\upalpha}\mu = 0$ where $\upalpha \equiv \upalpha_{[u,\mu]}(\cdot)$ is as above.

The (\textit{control}) parameters $\alpha$  take values in a compact set $A$ of $\mathds{R}^{k}$ for some positive $k$.  The case where $H$ is given with a $\max$ (instead of a $\min$) can be obtained analogously (see \cite{kouhkouhPhD} for further details).

The differential operator $\mathcal{L}_{\alpha}$ can be interpreted as the infinitesimal generator of the controlled stochastic process
\begin{equation}\label{eq: SDE intro}
    \text{d}X_{t} = b(X_{t},\alpha_{t})\text{d}t + \sqrt{2}\varrho(X_{t},\alpha_{t})\text{d}B_{t}
\end{equation}
where $B_{t}$ is a Wiener process while $f$ is the payoff integrand in a stochastic control problem. 
Note that \eqref{eq: SDE intro} should be understood in its weak sense (see  \cite{krylov37selection, krylov1969ito, lee2022analytic}).

We denote by $\mathcal{M}(\mathds{R}^{m})$ (respec. $\mathcal{M}^{+}(\mathds{R}^{m})$) the space of totally finite signed (respec. non-negative) Borel measures on $\mathds{R}^{m}$.  
We equip $\mathcal{M}(\mathds{R}^{m})$ with the \textit{Total-Variation} (TV) norm 
$\|\mu\|_{TV} = |\mu|(\mathds{R}^{m})$ where $|\mu|=\mu^{+}+\mu^{-}$ and $\mu^{+},\mu^{-}$ are the positive and negative parts of $\mu$.  
For $d\geq 1$,  $\mathcal{M}_{d}(\mathds{R}^{m})$ is the subset of measures with finite $d$-moment, i.e. for any $\mu\in \mathcal{M}_{d}(\mathds{R}^{m})$, one has $\int_{\mathds{R}^{m}}|x|^{d}\,\text{d}|\mu|(x)<+\infty$, and by $\mathcal{M}_{d}^{+}(\mathds{R}^{m})$ the subspace of non-negative measures. We denote by $\mathcal{P}(\mathds{R}^{m})$ the subset of probability measures and define $\mathcal{P}_{d}(\mathds{R}^{m})=\mathcal{P}(\mathds{R}^{m})\cap \mathcal{M}_{d}(\mathds{R}^{m})$. We write  for any  $g:\mathds{R}^{m}\to \mathds{R}$ measurable and  $\mu\in\mathcal{M}(\mathds{R}^{m})$
\begin{equation*}
    \langle g(\cdot)\, , \mu\rangle = \int_{\mathds{R}^{m}}g(x)\text{d}\mu(x).
\end{equation*}
We also recall the usual notations: if a measure $\mu$ has a density $\rho$ with respect to Lebesgue measure that we denote by $dx$, then $\mu$ is absolutely continuous with respect to $dx$, we write $\mu \ll dx$ and $\rho=\frac{d\mu}{dx}$ is the Radon-Nikodym derivative of $\mu$ with respect to $dx$. With slight abuse of notation, an element $\mu\in\mathcal{M}(\mathds{R}^{m})$ will denote either a measure or a density (when exists). Let $W^{p,k}(\mathds{R}^{m})$, $p\geq 1, k\geq 0$ be the standard Sobolev space of functions whose generalized derivatives up to order $k$ are in $L^{p}(\mathds{R}^{m})$. When we consider a measure $\mu$ instead of Lebesgue, we write $W^{p,k}(\mathds{R}^{m};\mu)$ or $L^{p}(\mathds{R}^{m};\mu)$ to denote the weighted Sobolev or Lebesgue space respectively. And let $W^{p,k}_{loc}(\mathds{R}^{m})$ be the class of functions such that $\chi f\in W^{p,k}(\mathds{R}^{m})$ for each $\chi \in C^{\infty}_{0}(\mathds{R}^{m})$ the class of infinitely differentiable functions with compact support in $\mathds{R}^{m}$.

Our \textbf{main result} (see Theorem \ref{thm: nec and suf - min}) can be informally stated as:\\ 
\textit{
Under assumptions that we shall soon make precise, the following are equivalent:
\begin{enumerate}[label = (\Roman*)]
    \item There exists a pair $(q_{\circ},\upalpha_{\circ})$ such that
    \begin{equation*}
        (q_{\circ},\upalpha_{\circ}) \in \argmin\limits_{\substack{q\in\mathcal{M}_{d}^{+}(\mathds{R}^{m})\\ \upalpha(\cdot)\in\mathcal{A}}}\big\{\langle f(\cdot\,,\upalpha(\cdot),q),\,q\,\rangle\, ,\;\text{s.t.: } 1-\langle 1,q\rangle = 0 \text{ and } q\in \text{Ker}(\mathcal{L}^{*}_{\upalpha})\big\}
    \end{equation*}
where $\mathcal{A}$ is the set of measurable functions from $\mathds{R}^{m}$ to $A$.
    \item{
        There exist $(c_{\circ},u_{\circ},q_{\circ})\in\mathds{R}\times W^{r,2}_{\text{loc}}(\mathds{R}^{m})\times W^{s,1}_{\text{loc}}(\mathds{R}^{m})$ for any $r\geq 1$, $s>m$ and a measurable function $\upalpha_{\circ}(\cdot):\mathds{R}^{m}\to A$, 
        that solve the MFG system
        \begin{equation*}
        \left\{
            \begin{aligned}
            & \quad \min\limits_{\alpha\in A}\{ -\text{trace}\big( a(x,\alpha)D^{2} u_{\circ}(x)\big) - b(x,\alpha)\cdot\nabla u_{\circ}(x) +  f(x,\alpha,q_{\circ}) \} = c_{\circ}  \\
            & - \text{trace}\big(D^{2}(a(x,\upalpha_{\circ}(x))q_{\circ}(x))\big) + \text{div}\big(b(x,\upalpha_{\circ})q_{\circ}(x)\big) = 0, \quad \quad  \text{a.e. in } \mathds{R}^{m}
            \end{aligned}
        \right.
        \end{equation*}
         such that the constant $c_{\circ}$ is defined by $c_{\circ}=\langle f(\cdot\,,\upalpha_{\circ}(\cdot),q_{\circ}),q_{\circ} \rangle$ and the function $\upalpha_{\circ}$ satisfies $\upalpha_{\circ}(x)\in \argmin\limits_{\alpha\in A}\{-\mathcal{L}_{\alpha}u_{\circ}(x) + f(x,\alpha,q_{\circ})\}$ a.e. in $\mathds{R}^{m}$.
    }
\end{enumerate}
}
In particular, the MFG system admits a solution in the sense of (II) above, that is in general not unique (see Corollary \ref{cor: main 2} and Remark \ref{rem: uniq}). 

Throughout the manuscript, we will need three sets of assumptions: \textit{measurability} assumptions (A), \textit{continuity} assumptions (B), \textit{differentiability} assumptions (C), and will refer to them wherever it is needed

\subsection{Assumptions: first batch}\label{sec:first batch}

We now state assumptions (A) and (C), and we keep assumptions (B) for soon after in \S\ref{sec:second batch}. 

\textit{Assumptions (A):} The 
(measurable) dependence on $x$.

\begin{description}
    \item[A1]{
    \begin{enumerate}[label = (\roman*)]
        \item $a=(a^{ij}_{\alpha})$ is a continuous mapping (uniformly in $\alpha$) on $\mathds{R}^m$ such that $a(x,\alpha)=\varrho(x,\alpha)\varrho(x,\alpha)^{\top}$ where $\varrho$ is a continuous in $x$ (unif. in $\alpha$) $m\times m_{1}$ matrix function (for some $m_{1}\geq m$),
        \item $b=(b^{i}_{\alpha}):\mathds{R}^m\times A\to \mathds{R}^m$ is a loc. bounded Borel-meas. vector field.
    \end{enumerate}}
    \item[A2]{For $p>m$, $a^{ij}(\cdot,\alpha)\in W^{p,1}_{\text{loc}}(\mathds{R}^{m})$ and $b^{i}(\cdot,\alpha)\in L^{p}_{\text{loc}}(\mathds{R}^{m})$, uniformly in $\alpha\in A$.   }
    \item[A3]{There exist $\overline{\Lambda}\,\geq\,\underline{\Lambda}\;>0$ such that $\forall\;x,\xi\in\mathds{R}^m$, $\;\underline{\Lambda} \|\xi\|^{2}\,\leq\, \xi\cdot a(x,\alpha)\xi \,\leq\, \overline{\Lambda}\|\xi\|^{2}$, \\ 
    uniformly in $\alpha\in A$, i.e. $(a^{ij})$ is positive, unif. bounded and nondegenerate.}
    \item[A4]{ The drift $b$ satisfies, for some positive numbers $\chi, \gamma_{1},\gamma_{2}$,
    \begin{equation*}
        \sup\limits_{\alpha\in A} \, b(x,\alpha)\cdot x  \leq \gamma_{1} - \gamma_{2}|x|^{\chi},\quad \forall\, x\in\mathds{R}^{m}
    \end{equation*}}
    \item[A5]{$f:\mathds{R}^{m}\times A\times \mathcal{M}(\mathds{R}^{m})\to \mathds{R}$ is such that
    \begin{enumerate}[label = (\roman*)] 
        \item $x\mapsto f(x,\alpha,\mu)$ is Borel-measurable on $\mathds{R}^{m}$,
        \item $f(\cdot\,,\alpha,\mu)\in L^{1}(\mathds{R}^{m};\mu)$, uniformly in $\alpha$ and for every $\mu\in\mathcal{M}^{+}_{d}(\mathds{R}^{m})$,  $d\geq 1$, such that $\mu\ll dx$,
    \end{enumerate}
    }
    \item[A6]{$\exists \, K_{b}>0$ and $\theta\in [0,d]$ such that $|b(x,\alpha)|\leq K_{b}(1+|x|)^{\theta}$ for all $x\in\mathds{R}^{m},\alpha\in A$.}
\end{description}

In assumption (A5-(ii)), what we are asking for is a polynomial growth of $f$ in $x$ of order at most $d$, since $\mu$ here is taken among $\mathcal{M}_{d}(\mathds{R}^{m})$. 
This is in fact a subexponential growth since $d\geq 1$ can be arbitrarily chosen. One can still handle an exponential growth provided  (A4) is strengthened (see \cite[Remark 2.12]{kouhkouh1}).

\textit{Assumptions (C): } The (differentiable) dependence on $\mu$.

\begin{description}
    \item[C1] $\mu\mapsto f(x,\alpha,\mu)$ has a \textit{Fréchet} (or strong) directional derivative at every $\mu\in \mathcal{M}^{+}_{d}(\mathds{R}^{m})$ such that $\mu\ll dx$, i.e.
        \begin{equation*}
            f(x,\alpha,\mu+h) = f(x,\alpha,\mu) + D_{\mu}f(x,\alpha,\mu)[h] + o(\|h\|),\quad \forall\,h\in\mathcal{M}_{d}(\mathds{R}^{m})
        \end{equation*}
        where $D_{\mu}f(x,\alpha,\mu)[h]$ is a bounded linear continuous functional of $h$ and $\|h\|$ is its TV-norm, and satisfies moreover
        \begin{equation*}
            \sup\limits_{\alpha\in A} \|D_{\mu}f(\cdot\,, \alpha,\mu)\|_{_{\text{op}}} \in L^{1}(\mathds{R}^{m};\mu),\quad \forall \, \mu \in \mathcal{M}_{d}^{+}(\mathds{R}^{m})
        \end{equation*}
        where $\|\cdot\|_{_{\text{op}}}$ is the operator norm. 
    \item[C2] 
    The Fréchet directional derivative of $f$ in $\mu\in\mathcal{M}^{+}_{d}(\mathds{R}^{m})$ such that $\mu\ll dx$ satisfies on the subset $\mathcal{M}_{d}^{+}(\mathds{R}^{m})$, uniformly in $\alpha$,
    \begin{equation*}
        \langle\, D_{\mu}f(\cdot\,,\alpha,\mu)[h]\,,\,\mu\, \rangle \leq 0, \quad \forall\,h\in\mathcal{M}^{+}_{d}(\mathds{R}^{m}).
    \end{equation*}
\end{description}



\textbf{Notation. } We shall keep the same notation $f(x,\alpha,\mu)$ whether $f$ depends on $\mu$ in a \textit{local} way, i.e. when we have $f(x,\alpha,\mu(x))$ defined on $\mathds{R}^{m}\times A \times \mathds{R}$, or $f$ depends on $\mu$ in a \textit{non-local} way, i.e. when we have $f(x,\alpha,\mu)$ defined on $\mathds{R}^{m}\times A\times \mathcal{M}(\mathds{R}^{m})$, having in mind that one can represent (in the \textit{local} case) $\mu(x)$ as a convolution with a Dirac measure with unit mass concentrated at zero, i.e. $\delta_{0}\ast \mu(x)$.





With assumption (C1), $\mu\mapsto f(x,\alpha,\mu)$ is \textit{Fréchet}-differentiable, hence there also exists a \textit{G\^{a}teaux} (directional) derivative, and we have
\begin{equation*}
    \lim\limits_{t\downarrow 0} t^{-1}(f(x,\alpha,\mu+th) - f(x,\alpha,\mu)) = D_{\mu}f(x,\alpha,\mu)[h],\quad \forall\, h \in\mathcal{M}_{d}(\mathds{R}^{m}).
\end{equation*}
In particular, $\mu\mapsto f(x,\alpha,\mu)$ is continuous and locally Lipschitz at every $\mu$ in the TV-norm.

Note that assumption (C2) is 
different from the monotonicity assumption discovered by Lasry and Lions \cite{lasry2007mean}, usually present in the MFG literature \cite{cardaliaguet2012long, cardaliaguet2013long} and which is (with our notations) 
\begin{equation}
\label{monotinicity LL}
\tag{M}
    \int_{\mathds{R}^{m}} \big(f(x,\alpha,\mu_{1}) - f(x,\alpha, \mu_{2})\big)\,\text{d}(\mu_{1}-\mu_{2})(x) \leq 0,\quad \forall \,\mu_{1},\mu_{2}\in\mathcal{M}_{d}(\mathds{R}^{m}).
\end{equation}
Setting $\mu_{1}=\mu + h$ and $\mu_{2}=\mu$, and assuming $f$ is Fréchet differentiable in the $\mu$-variable, then one gets
\begin{equation*}
    \int_{\mathds{R}^{m}} \big(f(x,\alpha,\mu+h) - f(x,\alpha, \mu)\big)\,\text{d}h(x) = \langle\,D_{\mu}f(\cdot\,,\alpha,\mu)[h]\,,\,h\,\rangle + o(\|h\|^{2})
\end{equation*}
and condition \eqref{monotinicity LL} hence implies
\begin{equation}
    \label{monotonicity LL 2}
    \tag{M'}
    \langle\,D_{\mu}f(\cdot\,,\alpha,\mu)[h]\,,\,h\,\rangle \leq 0,\quad \forall\,\mu,h\in\mathcal{M}_{d}(\mathds{R}^{m}).
\end{equation}
We shall discuss this later in Remark \Ref{rmk: assumption LL} in \S \ref{sec: rmk and ex}.



\subsection{Related results}

The theory of mean-field games started with the seminal papers by Lasry and Lions \cite{lasry2006jeux,lasry2006jeux2,lasry2007mean} and by Huang, Caines, Malhamé \cite{huang2006large}. Since then there is a huge literature on MFGs. For ergodic MFGs, we would like to refer to \cite{cesaroni2019introduction} and the many references therein. However many of the existing results consider bounded domains (mainly the torus), and very few treat the problem in the whole space. For the periodic case, we refer to \cite{bardi2016nonlinear, dragoni2018ergodic, feleqi2013derivation} that use PDE techniques for elliptic and subelliptic problems. The linear-quadratic setting is studied in \cite{bardi2014linear} where the solvability of the MFG system is reduced to the solvability of an algebraic Riccati equation and a Sylvester equation which also allow to get (at least in some examples) explicit solutions. In \cite{cesaroni2018concentration}, existence of classical solutions is proved in the whole space $\mathds{R}^{m}$ for ergodic MFGs of the form
\begin{equation*}
\left\{\; 
\begin{aligned}
& -\varepsilon \Delta u + H(Du) + c = g(m) + V(x)\\
& -\varepsilon \Delta q -\text{div}(qDH(Du)) = 0, \text{ in } \mathds{R}^{m}
\end{aligned}
\right.
\end{equation*}
where the potential $V$ is assumed to be coercive and $g$ is a local coupling term. The Hamiltonian $H$ also satisfies some growth assumptions. Their approach is variational based on the analysis of the non-convex energy associated to the system.  Another work in this same vein is the one in \cite{cirant2016stationary} where the coupling term is local, decreasing and unbounded satisfying some growth conditions. In this case, existence and non-existence results are shown using Sobolev regularity of the invariant measure and a blow-up procedure, and additional results in the case where the coupling term is local and increasing are also proven. A recent paper is \cite{bernardini2023ergodic} (see also \cite{bernardini2022mass}) which studies ergodic mean-field games in the whole space $\mathds{R}^{m}$ with a coercive potential $V$ and an attractive nonlocal coupling $g$ of Choquard-type. 
In the latter references, the setting is (with our notations) $a=I$ identity matrix, $b(x,\alpha)=\alpha$, $f(x,\alpha,q)=H^{*}(\alpha) - V(x) -  g(q)$ where $H^*$ is the Legendre transform of $H$ which is usually assumed to behave as a power $H(p) = \frac{1}{\gamma}|p|^{\gamma}$ (and hence also $H^{*}$). They are also concerned with \textit{classical solutions} whereas in the present manuscript we shall be interested in \textit{weak solutions}. 

Another paper with a setting that is closer to ours is \cite{arapostathis2017solutions}. Their setting is the one of (ergodic) stochastic control: the drift $b=b(x,\alpha)$ and the diffusion term $\varrho=\varrho(x)$ in \eqref{eq: SDE intro} are locally Lipschitz with an affine growth and local non-degeneracy, and the running cost $f$ satisfies some growth conditions. They proved existence of a solution to the MFG system and also studied the long time behavior. Their approach is based on the ergodic control formulation and relies on regularity of set-values maps corresponding to ergodic occupation measures  together with an application of Kakutani-Fan-Glicksberg fixed point theorem and convex analytic tools. 



Our method seems to be new in this regard. It relies on optimization on abstract Banach spaces, taking advantage of existing results in the theory of Dirichlet forms and diffusion operators. We shall also work with the Total-Variation norm, and not the Wasserstein metric as it is customary; see Remark \ref{rmk: TV}. Moreover, our results rely on a monotonocity assumption that is different from the one discovered by Lasry and Lions. Finally, let us mention that our assumptions concern the coefficients of the diffusion operator (or the underlying stochastic differential equation) rather than the  Hamiltonian.

The manuscript is organized as follows. In \textbf{Section \ref{sec: survey}} we provide the main known results from duality theory and also from diffusion operators, in particular we define the closed extension of an operator and which is the definition we shall consider for $\mathcal{L}_{\alpha}$ in the equation \eqref{eq: mfg - intro}. Then in \textbf{Section \ref{sec: preliminary results}} we prove preliminary results that will be needed throughout the manuscript. We also define the \textit{primal} and \textit{dual} optimization problems. We will then be ready in \textbf{Section \ref{sec: main results}} to state and prove the main result for ergodic MFGs, before we conclude with remarks about our assumptions and some examples.




\section{Survey of known results}\label{sec: survey} 

\subsection{On duality in optimization}\label{sec: duality}

We borrow from \cite{bonnans2013perturbation} some results on duality and optimization that are instrumental in our approach. 

Let $(X,X^{*})$ and $(Y,Y^{*})$ be paired spaces, i.e. such that each space of a pair is a locally convex topological vector space and is the topological dual of the other. We assume moreover that $X$ and $Y$ are Banach spaces that we endow with their respective strong topologies, while $X^{*}$ and $Y^{*}$ are endowed with the respective weak-$*$ topologies.

Let $Q$ be a closed convex subset of $X$ and $K$ a closed convex cone subset of $Y$. We are interested in first order optimality conditions for the optimization problem
\begin{equation}
\label{Primal - 1}
\tag{$P$}
    val(P) =\; \min\limits_{x\in Q} f(x),\quad \text{s.t.:}\;\;\; G(x)\in K
\end{equation}
where $f:X\to \mathds{R}$ and $G:X\to Y$. The objective function in \eqref{Primal - 1} can be reformulated as $f(x) + I_{Q}(x)$ while we minimize over the whole set $X$. We denote by $I_{Q}(\cdot)$ the indicator function ($I_{Q}(x) = 0$ if $x\in Q$, and $+\infty$ if $x\notin Q$). The Lagrangian of \eqref{Primal - 1} is 
\begin{equation}
\label{lagrangian}
    L(x,y^{*}) := f(x) + \langle y^{*}, G(x) \rangle,\quad (x,y^{*}) \in X\times Y^{*}.
\end{equation}

We embed the problem \eqref{Primal - 1} into the family of optimization problems
\begin{equation}
\label{Primal - y}
\tag{$P_{y}$}
    \min\limits_{x\in Q} f(x),\quad \text{s.t.:}\;\;\; G(x)+y\in K
\end{equation}
where $y\in Y$ is viewed as the parameter vector. Clearly for $y=0$, the corresponding problem $(P_{0})$ coincides with the problem \eqref{Primal - 1}. Let $v(y)$ be the corresponding value function
\begin{equation*}
    v(y) = val(P_{y}) = \inf\limits_{x\in Q}\; f(x) + I_{K}(G(x)+y). 
\end{equation*}

The (conjugate) dual of \eqref{Primal - 1} can be written in the form (see \cite[\S 2.5.3, p. 107]{bonnans2013perturbation}):
\begin{equation}
    \label{Dual}
    \tag{$D$}
    val(D) =\; \max\limits_{y^{*}\in Y^{*}} \big\{\inf\limits_{x\in Q} \;L(x,y^{*})\; - I^{*}_{K}(y^{*})\;  \big\}
\end{equation}
and $I^{*}_{K}(\cdot)$ is the Legendre-Fenchel conjugate of the indicator function supported on $K$, which is known as the \textit{support function} of the set $K$. 

Recall that $val(P)\geq val(D)$ (this can be easily obtained for example as a consequence of conjugate duality; see \cite[eq. (2.268), p.  96]{bonnans2013perturbation}, or by Lagrange duality; see \cite[Proposition 2.156, p.  104]{bonnans2013perturbation}) and that if for some $x_{o}\in Q$, $y_{o}^{*}\in Y^{*}$ the equality of primal and dual objective functions holds, i.e.
\begin{equation}
    \label{equality objective functions}
    f(x_{o}) + I_{K}(G(x_{o})) = \inf\limits_{x\in Q}\;L(x,y^{*}_{o}) - I^{*}_{K}(y_{o}^{*}),
\end{equation}
then $val(P)=val(D)$, and if the common value is finite, then $x_{o}\in Q$ and $y_{o}^{*}\in Y^{*}$ are optimal solutions of \eqref{Primal - 1} and \eqref{Dual} respectively. The equality \eqref{equality objective functions} can be written in the following equivalent form
\begin{equation}
    \label{equality objective functions - FY}
    \big(L(x_{o},y_{o}^{*})-\inf\limits_{x\in Q}L(x,y_{o}^{*})\big) + \big(I_{K}(G(x_{o})) + I^{*}_{K}(y_{o}^{*})-\langle y_{o}^{*},G(x_{o})\rangle\big) = 0.
\end{equation}
Clearly, the first term in the left hand side is non-negative and the second term is also non-negative by the Young-Fenchel inequality. Moreover the equality
\begin{equation*}
    I_{K}(G(x_{o})) + I^{*}_{K}(y_{o}^{*})-\langle y_{o}^{*},G(x_{o})\rangle = 0
\end{equation*}
holds if and only if $y_{o}^{*}\in \partial I_{K}(G(x_{o}))$; the subdifferential of the indicator function evaluated in $G(x_{o})$. 
Thus, the equality in \eqref{equality objective functions} is equivalent to
\begin{equation}
\label{optimality conditions - 1}
    x_{o}\in \argmin\limits_{x\in Q}L(x,y_{o}^{*})\quad \text{and}\quad y_{o}^{*}\in\partial I_{K}(G(x_{o})).
\end{equation}
And we have $\partial I_{K}(G(x_{o}))=N_{K}(G(x_{o}))$ where $N_{K}(\cdot)$ is the normal cone\footnote{If $S\subset X$ convex, then $N_{S}(x):=\{x^{*}\in X^{*}\,: \langle x^{*},z-x\rangle \leq 0\,\forall\;z\in S\}$. If $x\notin S$ then $N_{S}(x)=\emptyset$.} to $K$. Moreover, since $K$ is a convex cone, the condition $y_{o}^{*}\in N_{K}(G(x_{o}))$ is equivalent to
\begin{equation}
    \label{equiv cond normal cone}
    G(x_{o})\in K,\quad y_{o}^{*}\in K^{-}\quad \text{and}\quad \langle y_{o}^{*},G(x_{o}) \rangle = 0
\end{equation}
where $K^{-}$ is the polar (negative dual) cone\footnote{Let $C$ be a subset of $X$, then $C^{-}:=\{x^{*}\in X^{*}\;:\; \langle x^{*},x \rangle \leq 0,\quad\forall\;x\in C \}$.} of $K$. The optimality conditions can therefore be written as
\begin{equation}
\label{optimality conditions - 2}
	x_{o}\in \argmin\limits_{x\in Q}L(x,y_{o}^{*}),\quad G(x_{o})\in K,\quad y_{o}^{*}\in K^{-}\quad \text{and}\quad \langle y_{o}^{*},G(x_{o})=0.
\end{equation}


We are interested in existence of dual variables and in \textit{no duality gap} between \eqref{Primal - 1} and \eqref{Dual}, i.e. $val\eqref{Primal - 1}=val\eqref{Dual}$. We consider the convex case that we now define before stating the existence theorem.

\begin{definition}(\cite[Definition 2.163, p. 110]{bonnans2013perturbation})\label{def: convex}
We say that the problem \eqref{Primal - 1} is convex if the function $f(x)$ is convex, the set $Q$ is convex, the set $K$ is convex and closed, and the mapping $G(x)$ is convex with respect to the set\footnote{The mapping $G$ is convex w.r.t. the set $C$ if the multifunction $G(x)+C$ is convex (see \cite[Definition 2.103, p.72]{bonnans2013perturbation}), that is, for any $x_{1},x_{2}\in X$ and $t\in [0,1]$,
\begin{equation*}
    tG(x_{1}) + (1-t)G(x_{2}) - G(tx_{1} + (1-t)x_{2}) + C \subset C.
\end{equation*}
} $C:=-K$.
\end{definition}

\begin{theorem}\label{thm: equivalence}
    If \eqref{Primal - 1} is convex, $f,G$ are continuous, $Q$ is nonempty and closed,  val\eqref{Primal - 1} is finite and the following condition 
    \begin{equation}
    \label{eq: interior}
        0\in \text{\upshape int}\{G(Q)-K\}
    \end{equation}
    is satisfied, then there is no duality gap, and a feasible point $x_{\circ}$ is optimal if and only if there exists $\Bar{y}^{*}\in Y^{*}$ satisfying the conditions \eqref{optimality conditions - 2}. \\
    Moreover, if $x_{\circ}$ is an optimal solution of \eqref{Primal - 1}, then the set of all $\Bar{y}^{*}$ satisfying optimality conditions \eqref{optimality conditions - 2} is nonempty, closed and convex and coincides with the set of optimal solution of the dual problem \eqref{Dual}, and hence is the same for any optimal solution of \eqref{Primal - 1}.
\end{theorem}

\begin{proof}[\textbf{Proof.}]
    The theorem corresponds to the statements \textit{(i)} and \textit{(iii)} of \cite[Theorem 3.4, p.148]{bonnans2013perturbation} where it is assumed that \eqref{Primal - 1} is \textit{calm}. The latter condition is implied by \eqref{eq: interior}  when \eqref{Primal - 1} is convex, $f,G$ are continuous, $Q$ is nonempty and closed and val\eqref{Primal - 1} is finite; see \cite[p.149]{bonnans2013perturbation}.
\end{proof}

\begin{remark}\label{rmk: duality}
    In fact, in the situation of Theorem \ref{thm: equivalence} and when an optimal solution $x_{\circ}$ for \eqref{Primal - 1} exists, the optimal solution set of the dual problem \eqref{Dual} is a nonempty, convex, bounded and weak-$*$ compact subset of $Y^{*}$; see \cite[Theorem 3.6, p.149]{bonnans2013perturbation}. 
\end{remark}

The next proposition characterizes \eqref{eq: interior} in a particular case: \\
Let $Y$ be the Cartesian product of two Banach spaces $Y_{1}$ and $Y_{2}$, and $K=K_{1}\times K_{2}\subset Y_{1}\times Y_{2}$ where $K{1}$ and $K_{2}$ are closed convex subsets of $Y_{1}$ and $Y_{2}$ respectively. Let $G(x)=(G_{1}(x), G_{2}(x))$ with $G_{i}(x)\in Y_{i},i=1,2$. 
\begin{proposition}\label{prop: interior}
If $Y_{2}=X$, $G_{2}(x)=x$ for all $x\in X$ and $G(x)$ is $(-K)$-convex and continuously differentiable, then the following condition is equivalent to \eqref{eq: interior}
\begin{equation}
    \label{eq: interior 2}
    0\in \text{\upshape int}\{G_{1}(x_{\circ}) + DG_{1}(x_{\circ})[K_{2} - x_{\circ}]  - K_{1}\}
\end{equation}
at every feasible point $x_{\circ}\in \{x\in X\,:\, x\in Q \text{ and } G(x) \in K\}$.
\end{proposition}

\begin{proof}[\textbf{Proof}] 
Using \cite[Proposition 2.104, p.73]{bonnans2013perturbation}), we have \eqref{eq: interior} is equivalent to Robinson's constraint qualification
\begin{equation}
    \label{Robinson 1}
    0 \in \text{\upshape int}\{G(x_{\circ}) + DG(x_{\circ})(Q-x_{\circ})  -K\}
\end{equation}
which in turn is equivalent to \eqref{eq: interior 2}; see \cite[equation (2.192), p.71]{bonnans2013perturbation}).
\end{proof}

\subsection{Optimization in space of measures}\label{sec: opt space measure}

We consider a particular case of the optimization problem \eqref{Primal - 1}  that we write in the context of functionals depending on a measure following the results in \cite{molchanov2000tangent}.

In this subsection, we choose $Q$ and $K$ as closed convex subsets of $\mathcal{M}^{+}(\mathds{R}^{m})$ and a Banach space Y, respectively, and we define $f : \mathcal{M}(\mathds{R}^{m}) \to \mathds{R}$ and $G:\mathcal{M}(\mathds{R}^{m})\to Y$ as Fréchet differentiable functions. The derivative of $f$ is a linear functional $Df(\mu)[h]$ acting on $h\in \mathcal{M}(\mathds{R}^{m})$ and the derivative of $G$ is a linear operator $DG(\mu)[h]$ mapping $\mathcal{M}$ into $Y$. The optimization problem we consider is
\begin{equation}\label{Primal - measure - survey}
    \min f(\mu),\quad \text{ s.t.: }\quad \mu\in Q\; \text{ and }\; G(\mu)\in K
\end{equation}
We need now to define a notion of regularity (also called \textit{Constraint Qualification}) that is due to Robinson \cite{robinson1976first} (see also \cite[\S  2.3.4, p. 67]{bonnans2013perturbation}), analogue to \eqref{Robinson 1}.

\begin{definition}(\cite[Definition 1.1]{molchanov2000tangent}) \label{def: regular measure}
A measure $\mu$ is called regular for Problem \eqref{Primal - measure - survey} if 
\begin{equation}
    \label{eq: regular measure}
    0 \in int(G(\mu) + DG(\mu)[Q-\mu] - K)
\end{equation}
where $int(A)$ is the set of all $y\in A\subset Y$ such that $y+ty_{1}\in A$ for all $y_{1}\in Y$ and all sufficiently small positive $t$.
\end{definition}
The following theorem is \cite[Theorem 4.1]{cominetti1990metric} and gives first order necessary conditions for a minimum in Problem \eqref{Primal - measure - survey}. When applied to the framework of measures, it is stated in \cite{molchanov2000tangent}.
\begin{theorem}(\cite[Theorem 1.1]{molchanov2000tangent})\label{thm: nec cond measure opt}
Assume that both $f : \mathcal{M}(\mathds{R}^{m}) \to \mathds{R}$ and $G:\mathcal{M}(\mathds{R}^{m})\to Y$ are continuous on $Q$ and Fréchet differentiable at a regular $\mu_{\circ}\in Q$ such that $G(\mu_{\circ})\in K$. Then, if $\mu_{\circ}$ is a local minimum point in Problem \eqref{Primal - measure - survey}, the following (necessary) optimality condition is satisfied:
\begin{equation}
    \label{eq: optimality condition measure}
    Df(\mu_{\circ})[h] \geq 0,\quad \forall\, h\in T_{Q\cap G^{-1}(K)}(\mu_{\circ}),
\end{equation}
where $T_{B}(\mu)$ is the first order tangent set to a set $B$ at a point $\mu$ in a Banach space and is defined as
\begin{equation*}
    T_{B}(\mu) = \liminf\limits_{t\downarrow 0}\frac{B-\mu}{t}.
\end{equation*}
\end{theorem}
In order to make use of the latter theorem, we will need to determine what is the tangent set in the space of measures. In our case, we shall be interested in $\mathcal{M}^{+}(\mathds{R}^{m})$, the cone of finite non-negative measures.
\begin{theorem}(\cite[Theorem 2.1]{molchanov2000tangent})\label{thm: tangent set measure}
Let $\mu \in \mathcal{M}^{+}(\mathds{R}^{m})$. Then
\begin{equation}
    T_{\mathcal{M}^{+}(\mathds{R}^{m})}(\mu) = \{h \in \mathcal{M}(\mathds{R}^{m})\,:\; h^{-}\ll \mu \},
\end{equation}
where for a signed measure $h$, its Jordan decomposition is written as $h=h^{+}-h^{-}$, and for $p,q \in \mathcal{M}^{+}(\mathds{R}^{m})$, $p\ll q$ refers to absolute continuity of $p$ with respect to $q$.
\end{theorem}

A direct consequence of the latter theorem, is the case of measures with finite $d$-moment. It suffices indeed to replace $\mathcal{M}(\mathds{R}^{m})$ (respec. $\mathcal{M}^{+}(\mathds{R}^{m})$) with $\mathcal{M}_{d}(\mathds{R}^{m})$ (respec. $\mathcal{M}^{+}_{d}(\mathds{R}^{m})$ and obtain the following result.
\begin{corollary}\label{cor: tangent set measure}
Let $\mu \in \mathcal{M}^{+}_{d}(\mathds{R}^{m})$. Then
\begin{equation}
    T_{\mathcal{M}^{+}_{d}(\mathds{R}^{m})}(\mu) \supseteq \{h \in \mathcal{M}_{d}(\mathds{R}^{m})\,:\; h^{-}\ll \mu \}.
\end{equation}
\end{corollary}

\subsection{Extension of diffusion operators}\label{sec: diffusion} 
We resume in this subsection some known results from 
\cite{bogachev2002uniqueness} (see also \cite{bogachev2001regularity, bogachev2015fokker, bogachev2000generalization,  bogachev1999uniqueness, stannat1999nonsymmetric}). We shall be interested in a matrix-valued function $a=(a^{ij}_{\alpha})$ and a vector field $b=(b^{i}_{\alpha})$ such that $a^{ij}_{\alpha}(x) = a^{ij}(x,\alpha)$ and $b^{i}_{\alpha}(x) = b^{i}(x,\alpha)$ where $\alpha$ is some parameter in the compact set $A$. For the sake of simplicity of notations, we omit the dependence of $a$ and $b$ on the parameter $\alpha$, the latter being assumed fixed in the present subsection (its effect will be tackled in \S \ref{sec:second batch}). 
Hence we simply write $a=(a^{ij})$ a continuous mapping on $\mathds{R}^m$ and  $b=(b^i):\mathds{R}^m\to \mathds{R}^m$ a Borel-measurable vector field. Let us also set
\begin{equation}
\label{gen_diffusion op}
    L_{a,b}\varphi = a^{ij}\partial_{i}\partial_{j}\varphi + b^{i}\partial_{i}\varphi,\quad \varphi \in C^{\infty}_{0}(\mathds{R}^m),
\end{equation}
where we have used the standard summation rule for repeated indices. Assume $\mu$ is a locally finite (not necessarily non-negative) Borel measure on $\mathds{R}^m$, i.e. a measure on the Borel $\sigma$-algebra $\mathcal{B}(\mathds{R}^m)$ of $\mathds{R}^m$, solving the Fokker-Planck-Kolmogorov (FPK) equation
\begin{equation}
    \label{equation mu_diff op}
    L^{*}_{a,b}\mu = 0
\end{equation}
in the following sense:
\begin{equation}
\label{sense equation mu_diff op}
    a^{ij},b^{i}\in L^{1}_{\text{loc}}(\mathds{R}^{m};\mu) \quad \text{ and } \quad \int_{\mathds{R}^m}L_{a,b}\varphi\; d\mu = 0,\quad \forall\;\varphi\in C^{\infty}_{0}(\mathds{R}^m)
\end{equation}
Measures $\mu$ satisfying \eqref{equation mu_diff op} are called \textit{infinitesimally} invariant, or simply invariant if there is no confusion. And define
\begin{equation}
\label{measures in kernel}
    \mathcal{M}_{\text{ell}}^{a,b}:=\big\{\mu\;|\; \mu \text{ a probability measure on }\; \mathds{R}^{m}\;\text{satisfying }\; \eqref{equation mu_diff op}\big\}.
\end{equation}
where the subscript ``ell" stands for \textit{elliptic}. In \cite{bogachev1999uniqueness}, it is shown that the question whether or not $\mathcal{M}_{\text{ell}}^{a,b}$ contains at most one element turns out to be related to the question whether $\mu\in\mathcal{M}_{\text{ell}}^{a,b}$ is invariant for the $C_{0}$-semigroup generated by the closure of the operator $(L_{a,b},C^{\infty}_{0}(\mathds{R}^m))$. \\
In particular, and under assumptions that we will shortly make precise, if $\mathcal{M}_{\text{ell}}^{a,b} = \{\mu\}$ a singleton, then $\mu$ allows to define a new operator $(\overline{L}^{\mu}_{a,b}, D(\overline{L}^{\mu}_{a,b}))$ which is the closed extension of $(L_{a,b},C^{\infty}_{0}(\mathds{R}^m))$ on $L^{1}(\mathds{R}^{m};\mu)$.
The latter operator will play a key role in our main result on existence of solutions to \eqref{eq: mfg - intro}. 
The next theorem summarizes known results: their proofs and references can be found in \cite[\S 2.2]{kouhkouh1}. 

\begin{theorem}\label{thm summary diff}
The following statements hold.
\begin{enumerate}[label=\textbf{S.\arabic*},ref=S.\arabic*]
\item \label{thm existence inv meas} [Existence] Assume (A1), (A2), (A3) and (A4). Then $\mathcal{M}_{\text{ell}}^{a,b}$ as defined in \eqref{measures in kernel} is non-empty.
\item \label{thm regularity meas} [Regularity] Let $\mu$ be a locally finite and non-negative Borel measure satisfying \eqref{equation mu_diff op}. Assume (A1), (A2) and (A3). Then $\mu \ll dx$ with $\frac{d\mu}{dx}\in W^{p,1}_{\text{loc}}(\mathds{R}^{m})\big(\subset C^{1-\frac{m}{p}}(\mathds{R}^{m})\big)$. If $\rho$ denotes the continuous version of $\frac{d\mu}{dx}$, then for all compact $K\subset \mathds{R}^{m}$, $\exists\;c_{K}\in]0,\infty[$ s.t.: $\sup\limits_{K}\rho \leq c_{K}\inf\limits_{K}\rho$. In particular, either $\rho\equiv 0$ or $\rho(x)>0,\;\forall\;x\in\mathds{R}^{m}$.
\item \label{thm closed extension} [Extension] Assume (A1), (A2), (A3) and (A4). Then $\mathcal{M}^{a,b}_{\text{ell}}=\{\mu\}$ is a singleton and the following statements hold true
\begin{enumerate}[label=(\roman*)]
    \item there exists a closed extension of the operator $(L_{a,b},C^{\infty}_{0}(\mathds{R}^{m}))$ on $L^{1}(\mathds{R}^{m};\mu)$;
    \item its closure $(\overline{L}^{\mu}_{a,b}, D(\overline{L}^{\mu}_{a,b}))$ on $L^{1}(\mathds{R}^{m};\mu)$ generates a  $C_{0}$-semigroup $(T^{\mu}_{t})_{t\geq 0}$ on $L^{1}(\mathds{R}^{m};\mu)$;
    \item $(T^{\mu}_{t})_{t\geq 0}$ is the only $C_{0}$-semigroup on $L^{1}(\mathds{R}^{m};\mu)$ which has a generator extending $(L_{a,b},C^{\infty}_{0}(\mathds{R}^{m}))$;
    \item $(T^{\mu}_{t})_{t\geq 0}$ is contractive, and $\mu$ is $(T^{\mu}_{t})_{t\geq 0}$-invariant in the sense
    \begin{equation}\label{eq: semigroup invariance}
        \int_{\mathds{R}^{m}}T^{\mu}_{t}fd\mu = \int_{\mathds{R}^{m}}fd\mu,\quad \forall\;f\in L^{\infty}(\mathds{R}^{m};\mu).
    \end{equation}
\end{enumerate}
\end{enumerate}
\end{theorem}

Let us consider now the situation of Theorem \ref{thm summary diff}. 
Fix $\mu\in\mathcal{M}_{\text{ell}}^{a,b}$. As observed in \cite[\S 2.3]{bogachev1999uniqueness}, $\mu$ is equivalent to Lebesgue measure, and therefore is strictly positive on all non-empty open subsets of $\mathds{R}^{m}$. So $C^{\infty}_{0}(\mathds{R}^{m})$ can be identified with a subset of $L^{1}(\mathds{R}^{m};\mu)$, since each corresponding $\mu$-class has a unique continuous $\mu$-version. Hence the operator $(L_{a,b},C^{\infty}_{0}(\mathds{R}^{m}))$ is well defined on $L^{1}(\mathds{R}^{m};\mu)$. 

Thanks to this result, we can now define on a larger space the operator $\mathcal{L}$ in the problem \eqref{eq: mfg - intro}. This is an important step when dealing with unbounded right-hand side terms $f$ in \eqref{eq: mfg - intro}, since there cannot exist any solution in $C^{\infty}_{0}(\mathds{R}^{m})$.\\
Indeed, the differential operator $(\mathcal{L}, D(\mathcal{L}))$ in \eqref{eq: mfg - intro} should be understood in the sense of the closed extension $(\overline{L}^{\mu}_{a,b},D(\overline{L}^{\mu}_{a,b}))$ provided by (\ref{thm closed extension}) in Theorem \ref{thm summary diff}, where $D(\overline{L}^{\mu}_{a,b})$ is the closure of $C^{\infty}_{0}(\mathds{R}^m)$ in $L^{1}(\mathds{R}^{m};\mu)$. More precisely, we have $C^{\infty}_{0}(\mathds{R}^m)\subset D(\overline{L}^{\mu}_{a,b})\subset L^{1}(\mathds{R}^{m};\mu)$ with dense inclusions.

In the following, we state from \cite{bogachev2002uniqueness} a theorem which makes $D(\overline{L}^{\mu}_{a,b})$ more precise. In fact, for every $r\in [1,+\infty)$, the restriction of the semigroup $\{T^{\mu}_{t}\}_{t\geq 0}$, whose generator is $\overline{L}^{\mu}_{a,b}$, to $L^{r}(\mathds{R}^{m};\mu)$ is a strongly continuous semigroup on $L^{r}(\mathds{R}^{m};\mu)$ (see \cite[Lemma 5.1.4, p. 180]{bogachev2015fokker}). Its generator will be denoted by $(L^{\mu,r}_{a,b},D(L^{\mu,r}_{a,b}))$, where
\begin{equation*}
    D(L^{\mu,r}_{a,b}) = \{f\in D(L^{\mu}_{a,b})\cap L^{r}(\mathds{R}^{m};\mu): \; L^{\mu}_{a,b}f\in L^{r}(\mathds{R}^{m};\mu)\}.
\end{equation*}

\begin{theorem}(\cite[Theorem 2.8(i)]{bogachev2002uniqueness})\label{thm sobolev domain extension}
In the situation of (\ref{thm closed extension}) in Theorem \ref{thm summary diff}, one has for any $r\in [1,+\infty)$
\begin{equation}\label{eq: domain extension}
\begin{aligned}
    & (L^{\mu,r}_{a,b},D(L^{\mu,r}_{a,b})) \subset \{f\in L^{r}(\mathds{R}^{m};\mu)\cap W^{r,2}_{\text{loc}}(\mathds{R}^m)\,:\, L_{a,b}f\in L^{r}(\mathds{R}^{m};\mu)\}\\
    & \quad \text{and }\; L^{\mu,r}_{a,b}f=L_{a,b}f\quad \text{for all}\; f\in D(L^{\mu,r}_{a,b}).
\end{aligned}
\end{equation}
\end{theorem}


The next result is from \cite{veretennikov1997polynomial} and concerns the moments of the invariant measure. 
\begin{lemma}\label{conv-prob law}
Assuming (A1), (A3) and (A4), the invariant probability measure $\mu$ exists and has finite moments of any order $\ell\geq1$, i.e. $\int_{\mathds{R}^{m}}|x|^{\ell}\,\text{d}\mu(x) <+\infty$. 
\end{lemma}
\begin{proof}[\textbf{Proof}] See the proof of \cite[Lemma 1]{kouhkouh1} which is a particular case of the more general result in \cite[Theorem 6]{veretennikov1997polynomial} (in particular \cite[eq. (28) in \S 6]{veretennikov1997polynomial}). 
\end{proof}


We will also need an estimate from \cite{bogachev2018poisson} on the distance in TV-norm between two invariant measures. This is summarized in the companion paper \cite[\S 2.3]{kouhkouh1}.

\section{Preliminary results}\label{sec: preliminary results}

\subsection{An exchange property}

The following proposition, whose proof is in \cite{kouhkouh1}, allows us to exchange the order of the minimization (or maximization) with the integration with respect to a measure $q\in \mathcal{M}_{d}^{+}(\mathds{R}^{m})$, i.e. non-negative totally finite Borel measure with finite moment of order $d$.

\begin{proposition}\label{prop: exchange prop}
Let $f$ satisfies (A5). The following holds for any $q\in \mathcal{M}_{d}^{+}(\mathds{R}^{m})$
\begin{equation}
    \label{eq: exchange prop}
    \int_{\mathds{R}^{m}}\min\limits_{\alpha\in A} f(x,\alpha,q)\,\text{d}q(x) = \min\limits_{\upalpha(\cdot)\in\mathcal{A}}\int_{\mathds{R}^{m}}f(x,\upalpha(x),q)\,\text{d}q(x)
\end{equation}
where $A$ is a compact subset of $\mathds{R}^{k}$, for some $k>0$, and $\mathcal{A}$ is the set of measurable functions $\upalpha(\cdot):\mathds{R}^{m}\to A$. And the same holds true with $\max$ instead of $\min$.
\end{proposition}

The \textit{exchange property} in Proposition \ref{prop: exchange prop} ensures that we can exchange the minimization over the parameters $\alpha$ and the duality product in $\mathcal{M}_{d}(\mathds{R}^{m})$ provided we define the second argument in $f$ as measurable functions $\upalpha(\cdot)\in\mathcal{A} =L^{\infty}(\mathds{R}^{m},A)$ instead of vectors $\alpha\in A$, that is,
\begin{equation*}
    \min\limits_{\upalpha(\cdot)\in \mathcal{A}}\;\langle\; f(\cdot\,, \upalpha(\cdot),q)\;,\; q\;\rangle = \langle\; \min\limits_{\alpha\in A} f(\cdot\,,\alpha,q)\;,\; q\; \rangle
\end{equation*}

Before we move to the next preliminary result, we state the second part of our assumptions bearing in mind the analysis in \S \ref{sec: diffusion}. 

\subsection{Assumptions: second batch}\label{sec:second batch}

Along with assumptions (A) and (C) presented earlier in \S \ref{sec:first batch}, we will need some assumptions on the dependence of our data on the parameter $\alpha \in A$.

Recall $A$ a compact subset of $\mathds{R}^{k}$ for some $k>0$, and $\mathcal{A}:=L^{\infty}(\mathds{R}^{m},A)$ the set of measurable functions $\upalpha(\cdot):\mathds{R}^{m}\to A$. We need the following definition.

\begin{definition}\label{def: cont} 
Noting that $\mathcal{A}\subset L^{\infty}(\mathds{R}^{m})=(L^{1}(\mathds{R}^{m}))^{*}$, we say\\
    \textbullet \, the map $\upalpha(\cdot)\mapsto g(\cdot\,,\upalpha(\cdot)) \in  L^{\ell}(\mathds{R}^{m};\mu_{\upalpha})$, $\ell\geq 1$, is \underline{weak-$*$ continuous at $\upalpha$} if 
    \begin{equation*}
    \begin{aligned}
        &\forall\,\varepsilon>0,\; \exists\, \delta>0 \text{ and a finite collection } \{\xi_{1}, \dots,\xi_{n}\} \text{ from } L^{1}(\mathds{R}^{m}) \text{ such that }\\
        & \quad \; \forall\, \upbeta(\cdot)\in \mathcal{A} \text{ satisfying } 
        \left|\int_{\mathds{R}^{m}}(\upalpha(x) - \upbeta(x))\xi_{i}(x)\text{d}x\right|<\delta \text{ for } i=1, \dots,n, \\
        & \quad \quad \quad \quad \quad \text{ we have } \|g(\cdot\,,\upalpha(\cdot)) - g(\cdot\,,\upbeta(\cdot))\|_{L^{\ell}(\mathds{R}^{m};\mu_{\upalpha})}<\varepsilon
    \end{aligned}
    \end{equation*}
    \textbullet \, the functional $\upalpha(\cdot)\mapsto \mathcal{G}(\upalpha)\in \mathds{R}$ is \underline{weak-$*$ continuous at $\upalpha$} if in the last line of the above definition we have $|\mathcal{G}(\upalpha) - \mathcal{G}(\upbeta)|<\varepsilon$.\\
    \textbullet \, the map\slash functional is \underline{weak-$*$ continuous}, if it is weak-$*$ continuous at every $\upalpha$.
\end{definition}

Let us recall from \S \ref{sec: diffusion} existence, uniqueness and regularity of $\mu_{\upalpha}$ the invariant probability measure satisfying the FPK equation $\mathcal{L}^{*}_{\upalpha}\mu_{\upalpha}=0$ where $\mathcal{L}^{*}_{\upalpha}$ is the formal adjoint operator to second order elliptic operator
\begin{equation*}
    \mathcal{L}_{\alpha}\varphi(x) \coloneqq \text{trace}\big( a(x,\alpha)D^{2}\varphi(x)\big) + b(x,\alpha)\cdot\nabla\varphi(x).
\end{equation*}

\textit{Assumptions (B):} The (continuous) dependence on $\alpha$.

For all $1\leq i,j\leq m$, we assume
\begin{description}
    \item[B1]{The map 
    $\upalpha(\cdot)\mapsto f(\cdot\, , \upalpha(\cdot),\mu_{\upalpha})$ is weak-$*$ continuous from $\mathcal{A}$ to $L^{1}(\mathds{R}^{m};\mu_{\upalpha})$,
    }
    \item[B2]{
    The maps 
    $\upalpha(\cdot)\mapsto a^{ij}(\cdot\, , \upalpha(\cdot))$ is weak-$*$ continuous from $\mathcal{A}$ to $L^{4}(\mathds{R}^{m};\mu_{\upalpha})$,  
    }
    \item[B3]{
    The maps 
    $\upalpha(\cdot)\mapsto \partial_{x_{j}} a^{ij}(\cdot\, , \upalpha(\cdot))$ has a polynomial growth and is weak-$*$ continuous from $\mathcal{A}$ to $L^{2}(\mathds{R}^{m};\mu_{\upalpha})$. 
    The notation $\partial_{x_{j}} a^{ij}(\cdot\, , \upalpha(\cdot))$ means the derivative w.r.t. the $j$-th component of the first argument,
    }
    \item[B4]{
    The maps 
    $\upalpha(\cdot)\mapsto b^{i}(\cdot\, , \upalpha(\cdot))$ is weak-$*$ continuous from $\mathcal{A}$ to $L^{2}(\mathds{R}^{m};\mu_{\upalpha})$.
    }
\end{description}

\textbf{Note:} In assumption (B1) what we mean is the weak-$*$ continuity only in the second argument of $f$, although its third argument is evaluated in $\mu_{\upalpha}$ that depends on $\upalpha$. In the notation of the above-mentioned definition, assumption (B1) means\\
    \begin{equation*}
    \begin{aligned}
        &\forall\,\varepsilon>0,\; \exists\, \delta>0 \text{ and a finite collection } \{\xi_{1}, \dots,\xi_{n}\} \text{ from } L^{1}(\mathds{R}^{m}) \text{ such that }\\
        & \quad \; \forall\, \upbeta(\cdot)\in \mathcal{A} \text{ satisfying } 
        \left|\int_{\mathds{R}^{m}}(\upalpha(x) - \upbeta(x))\xi_{i}(x)\text{d}x\right|<\delta \text{ for } i=1, \dots,n, \\
        & \quad \quad \quad \quad \quad \text{ we have } \|f(\cdot\,,\upalpha(\cdot), \mu_{\upalpha}) - g(\cdot\,,\upbeta(\cdot), \mu_{\upalpha})\|_{L^{\ell}(\mathds{R}^{m};\mu_{\upalpha})}<\varepsilon.
    \end{aligned}
    \end{equation*}
And this should hold for every $\upalpha\in \mathcal{A}$.

In our setting, the assumptions (B) are satisfied for example when $\phi=a^{ij},\partial_{x_{j}}a^{ij},b^{i}$ is such that
\begin{equation}\label{assumption phi 1}
    |\phi(x,\alpha) - \phi(x,\beta)| \leq \psi(x) |\alpha - \beta|^{r},\quad \forall \, x\in\mathds{R}^{m},\, \alpha,\beta\in A
\end{equation}
and when $f$ satisfies
\begin{equation}\label{assumption phi 2}
    |f(x,\alpha,\mu) - f(x,\beta,\mu)| \leq \tilde{\psi}(x,\mu) |\alpha - \beta|^{\tilde{r}},\quad \forall \, x\in\mathds{R}^{m},\, \alpha,\beta\in A, \, \mu\in \mathcal{P}_{d}(\mathds{R}^{m})
\end{equation}
where $r, \tilde{r}>0$, and $\psi(\cdot)$ and $\tilde{\psi}(\cdot,\mu)$ have a polynomial growth; see \cite[\S 3]{kouhkouh1}. 

We shall also need an assumption that will play a crucial role in the validity of our method: besides the standing assumptions (A), (B) and (C), we denote again by the operator $(\mathcal{L}_{\upalpha},D(\mathcal{L}_{\upalpha}))$ its closed extension $(\overline{L}^{\mu}_{A,b},D(\overline{L}^{\mu}_{A,b}))$ as given by (\ref{thm closed extension}) in Theorem \ref{thm summary diff} and Theorem \ref{thm sobolev domain extension}, and we assume the following holds true

\begin{description}
    \item[(A*)]{The domain $D(\mathcal{L}_{\upalpha})$ of the closed extension is nonempty and independent of $\upalpha$.}
\end{description}

Examples where this assumption is satisfied are discussed in \cite[Remark 4.2]{kouhkouh1}. It means that there exists $\widetilde{\upalpha}(\cdot) \in \mathcal{A}$ such that for all $\upalpha(\cdot)\in \mathcal{A}$, one has $D(\mathcal{L}_{\upalpha}) = D(\mathcal{L}_{\widetilde{\upalpha}})$, and  $\mathcal{L}_{\widetilde{\upalpha}}$ falls in the framework of the previous sections, in particular it satisfies Theorem \ref{thm summary diff} and Theorem \ref{thm sobolev domain extension}. The nonemptiness assumption is trivial otherwise the PDE problem \eqref{eq: mfg - intro} does not make sense. We will hereafter denote by $D(\mathcal{L}_{0})$ the latter domain. 

\subsection{A continuity property}

We are here concerned with the continuity of the functional $F:\mathcal{A}\to \mathds{R}$ defined by 
\begin{equation}
    \label{functional F}
    F(\upalpha) := \int_{\mathds{R}^{m}}f(x,\upalpha(x),\mu_{\upalpha})\text{d}\mu_{\upalpha}(x) = \langle f(\cdot\,,\upalpha(\cdot),\mu_{\upalpha}),\mu_{\upalpha} \rangle
\end{equation}
where $\mathcal{A}$ is endowed with its weak-$*$ topology, and $\mu_{\upalpha}$ is the unique invariant probability measure satisfying the FPK equation $\mathcal{L}^{*}_{\upalpha}\mu_{\upalpha}=0$ where $\mathcal{L}^{*}_{\upalpha}$ is the formal adjoint operator to second order elliptic operator
\begin{equation*}
    \mathcal{L}_{\upalpha}\varphi(x) = \text{trace} \big( a(x,\upalpha(x))D\varphi(x) \big) + b(x,\upalpha(x))\cdot \nabla \varphi(x).
\end{equation*}
We recall that existence, uniqueness and regularity of $\mu_{\upalpha}$ have been discussed in \S \ref{sec: diffusion}. We need the matrix-norm: for a matrix function $M=(M^{ij})\in \mathds{R}^{p\times q}$, $p,q\geq 1$, we write $|M(x)|:=\max_{1\leq i\leq p}\sum_{j=1}^{q}|M^{ij}(x)|$ and $\| \,|M|\,\|^{\ell}_{L^{\ell}(\mathds{R}^{m};\mu)} = \int |M(x)|^{\ell}\text{d}\mu$.

For simplicity of notation, we write the functions $a_{\upalpha} = a(\cdot,\upalpha(\cdot))$, $b_{\alpha}=b(\cdot,\upalpha(\cdot))$ 
and the weighted Lebesgue space $L^{\ell}_{\mu_{\upalpha}} := L^{\ell}(\mathds{R}^{m};\mu_{\upalpha})$.

\begin{proposition}\label{prop: continuity F}
Assume (A), (B) and (C) are satisfied. Then the functional $F$ defined in \eqref{functional F} is weak-$*$ continuous (see Definition \ref{def: cont}). 
\end{proposition}

\begin{proof}[\textbf{Proof}]
Let $\upalpha\in \mathcal{A}$ be fixed. For a given $\varepsilon>0$, we want to find $\delta>0$ and a finite collection $\{\xi_{1},\dots,\xi_{n}\}$ from $L^{1}(\mathds{R}^{m})$ such that $|F(\upalpha)-F(\upbeta)|<\varepsilon$ holds 
\begin{equation}\label{def: weak star}
    \forall\, \upbeta(\cdot)\in \mathcal{A} \text{ satisfying } 
        \left|\int_{\mathds{R}^{m}}(\upalpha(x) - \upbeta(x))\xi_{i}(x)\text{d}x\right|<\delta \text{ for } i=1, \dots,n.
\end{equation}
We start by writing
\begingroup
\allowdisplaybreaks
    \begin{align*}
        F(\upalpha) - F(\upbeta) & = \langle f(\cdot\,,\upalpha(\cdot),\mu_{\upalpha}),\mu_{\upalpha} \rangle - \langle f(\cdot\,,\upbeta(\cdot),\mu_{\upbeta}), \mu_{\upbeta} \rangle \\
        & = \langle f(\cdot\,,\upalpha(\cdot),\mu_{\upalpha}) - f(\cdot\,,\upbeta(\cdot),\mu_{\upalpha}),\mu_{\upalpha} \rangle\\
        & \quad \quad + \langle f(\cdot\,,\upbeta(\cdot),\mu_{\upalpha}) - f(\cdot\,,\upbeta(\cdot),\mu_{\upbeta}),\mu_{\upalpha} \rangle\\
        & \quad \quad \quad \quad + \langle f(\cdot\,,\upbeta(\cdot),\mu_{\upbeta}), \mu_{\upalpha} - \mu_{\upbeta} \rangle \\
        & \leq \|f(\cdot\,,\upalpha(\cdot),\mu_{\upalpha}) - f(\cdot\,,\upbeta(\cdot),\mu_{\upalpha})\|_{_{L^{1}_{\mu_{\upalpha}}}}\\
        & \quad \quad + \|f(\cdot\,,\upbeta(\cdot),\mu_{\upalpha}) - f(\cdot\,,\upbeta(\cdot),\mu_{\upbeta})\|_{_{L^{1}_{\mu_{\upalpha}}}}\\
        & \quad \quad \quad \quad + \langle f(\cdot\,,\upbeta(\cdot),\mu_{\upbeta}), \mu_{\upalpha} - \mu_{\upbeta} \rangle.
    \end{align*}
\endgroup

The first term is controlled with assumption (B1): it ensures existence of $\delta_{1}>0$ and a finite collection $\mathfrak{S}_{1}:=\{\xi_{i}^{1},\dots,\xi_{n}^{1}\}$ from $L^{1}(\mathds{R}^{m})$ s.t. 
\begin{equation*}
    \|f(\cdot\,,\upalpha(\cdot),\mu_{\upalpha}) - f(\cdot\,,\upbeta(\cdot),\mu_{\upalpha})\|_{_{L^{1}_{\mu_{\upalpha}}}} \leq \frac{\varepsilon}{3}.
\end{equation*}

The third term is estimated in the proof of \cite[Proposition 3.4]{kouhkouh1} where we have shown 
\begin{equation}\label{eq: proof diff TV}
    \begin{aligned}
        & |\langle f(\cdot\,,\upbeta(\cdot),\mu_{\upbeta}), \mu_{\upalpha} - \mu_{\upbeta} \rangle| \leq C\,\|\mu_{\upalpha} - \mu_{\upbeta}\|_{_{TV}}^{\frac{1}{2}}\\
        & \quad \quad \quad  \leq C\left( \| \,|b_{\upalpha} - b_{\upbeta} |\,\|_{_{L^{2}_{\mu_{\upalpha}}}} 
    + \| \,|\nabla a_{\upalpha} - \nabla a_{\beta}|\, \|_{_{L^{2}_{\mu_{\upalpha}}}}
    + \| \,|a_{\upalpha} - a_{\upbeta}|\,\|_{_{L^{4}_{\mu_{\upalpha}}}}
    \right)^{\frac{1}{2}}
    \end{aligned}
\end{equation}
for some constant $C>0$ depending on the parameters in the assumptions, on the diffusion matrix $a_{\upbeta}$ and on $\mu_{\upalpha},\mu_{\beta}$. 
Then, assumptions (B2, B3, B4) ensure existence of $\delta_{2}>0$ and a finite collection $\mathfrak{S}_{2}$ from $L^{1}(\mathds{R}^{m})$ s.t. the r.h.s of the latter inequality is less or equal $\varepsilon/3$, hence
\begin{equation*}
    |\langle f(\cdot\,,\upbeta(\cdot),\mu_{\upbeta}), \mu_{\upalpha} - \mu_{\upbeta} \rangle| \leq \frac{\varepsilon}{3}.
\end{equation*}
Indeed, $\mathfrak{S}_{2}$ would be the union of the 3 finite collections from $L^{1}(\mathds{R}^{m})$ for which the three terms $\|\,|b_{\upalpha} - b_{\upbeta} |\,\|_{_{L^{2}_{\mu_{\upalpha}}}}, \| \,|\nabla a_{\upalpha} - \nabla a_{\beta}|\, \|_{_{L^{2}_{\mu_{\upalpha}}}}, \| \,|a_{\upalpha} - a_{\upbeta}|\,\|_{_{L^{4}_{\mu_{\upalpha}}}}$ are respectively less or equal $\frac{1}{3}\frac{\varepsilon^{2}}{(3C)^{2}}=:\Tilde{\varepsilon}>0$ and $\delta_{2}$ would be the minimum of the three $\delta$'s in the definition.

We are then left with estimating the second term. We have
\begin{equation*}
    \begin{aligned}
    \|f(\cdot\,,\upbeta(\cdot),\mu_{\upalpha})  - f(\cdot\,,\upbeta(\cdot),\mu_{\upbeta})\|_{_{L^{1}_{\mu_{\upalpha}}}} & = \int |f(x,\upbeta(x),\mu_{\upalpha})  - f(x,\upbeta(x),\mu_{\upbeta})|\,\text{d}\mu_{\upalpha}(x).
    \end{aligned}
\end{equation*}
Using (C1), and whenever $\|\mu_{\upalpha}-\mu_{\upbeta}\|_{TV}$ is small enough, we can write
\begin{equation*}
    f(x,\upbeta(x),\mu_{\upbeta})  - f(x,\upbeta(x),\mu_{\upalpha}) = D_{\mu}f(x,\upbeta(x),\mu_{\upalpha})[\mu_{\upbeta}-\mu_{\upalpha}] + o\big(\|\mu_{\upalpha}-\mu_{\upbeta}\|_{TV}\big)
\end{equation*}
and then
\begin{equation*}
\begin{aligned}
    |f(x,\upbeta(x),\mu_{\upalpha})  - f(x,\upbeta(x),\mu_{\upbeta})| & \leq 2\,\| D_{\mu} f(x,\upbeta(x),\mu_{\upalpha}) \|_{_{\text{op}}} \| \mu_{\upalpha} - \mu_{\upbeta}\|_{TV}\\
    & \leq 2\,\| \mu_{\upalpha} - \mu_{\upbeta}\|_{TV} \sup\limits_{\beta\in A}\| D_{\mu} f(x,\beta,\mu_{\upalpha}) \|_{_{\text{op}}}.
\end{aligned}
\end{equation*}
Integrating w.r.t $\mu_{\upalpha}$ yields
\begin{equation*}
\begin{aligned}
    \| f(\cdot\,,\upbeta(\cdot),\mu_{\upalpha})  - f(\cdot\,,\upbeta(\cdot),\mu_{\upbeta})\|_{L^{1}_{\mu_{\upalpha}}} \leq \Tilde{C}\, \| \mu_{\upalpha} - \mu_{\upbeta}\|_{TV} 
\end{aligned}
\end{equation*}
where $\Tilde{C}=2\big\|\sup\limits_{\beta\in A}\| D_{\mu} f(\cdot\,,\beta,\mu_{\upalpha}) \|_{_{\text{op}}}\big\|_{L^{1}_{\mu_{\upalpha}}}<\infty$ thanks to the second statement in assumption (C1). \\
But we have seen that $\| \mu_{\upalpha} - \mu_{\upbeta}\|_{TV}$ can be made indeed arbitrarily small, e.g. less or equal $\varepsilon/3$, as long as $\upbeta$ satisfies \eqref{def: weak star} for some finite collection $\mathfrak{S}_{3}$ and $\delta_{3}>0$, and this is true thanks to the assumptions (B2, B3, B4) together with the second inequality in \eqref{eq: proof diff TV}. 

To sum up, we choose  $\mathfrak{S}:=\mathfrak{S}_{1}\cup \mathfrak{S}_{2} \cup \mathfrak{S}_{3}$ as a finite collection in $L^{1}(\mathds{R}^{m})$, then we have $|F(\upalpha) - F(\upbeta)|<\varepsilon$, for all $\upbeta$ satisfying \eqref{def: weak star} with $\xi\in\mathfrak{S}$ and $\delta := 
\min(\delta_{1},\delta_{2},\delta_{3})$. Indeed, if $\upbeta$ satisfies \eqref{def: weak star} for all $\xi \in \mathfrak{S}$ and for such $\delta$, then in particular it satisfies \eqref{def: weak star} for all $\xi \in \mathfrak{S}_{i}$ and for $\delta_{i}\geq \delta$, $i\in \{1,2,3\}$. So the three terms estimated in the proof are less or equal $\varepsilon/3$ and we have the desired weak-$*$ continuity. 
\end{proof}

\subsection{The primal problem}\label{sec: primal prob}



We state our \textit{primal} problem as follows
\begin{equation}
    \label{eq: primal - min}
    \tag{$\mathfrak{P}$}
    \min\limits_{q\in\mathcal{M}_{d}^{+}(\mathds{R}^{m})}\left\{\,\min\limits_{\upalpha(\cdot)\in \mathcal{A}}\; \langle f(\cdot\,,\upalpha(\cdot),q),q\rangle,\quad \text{s.t.: } 1-\langle 1,q\rangle = 0 \;\text{and }\, q\in \text{Ker}(\mathcal{L}_{\upalpha}^{*})\right\}
\end{equation}
where we recall $\langle f(\cdot\,,\upalpha(\cdot),q),q\rangle = \int_{\mathds{R}^m}f(x,\upalpha(x),q)\text{d}q(x)$. For the convenience of the reader, we will use the same notation as in \S\ref{sec: duality}, that is,
\begin{gather*}
    X=\mathcal{M}_{d}(\mathds{R}^{m})\quad \text{and} \quad Q=\mathcal{M}_{d}^{+}(\mathds{R}^{m})\\
    G_{1}:X \to \mathds{R},\quad \text{s.t.}\quad G_{1}(q)=1-\langle 1,q\rangle\\
    G_{2}:X\to X,\quad \text{s.t.}\quad G_{2}(q)=q\\
    G=(G_{1},G_{2})\quad \text{and} \quad Y=\mathds{R}\times X\\
    K_{1}=\{0\},\; K_{2}(\upalpha)=\text{Ker}(\mathcal{L}^{*}_{\upalpha})\quad \text{and} \quad K_{\upalpha}=K_{1}\times K_{2}(\upalpha)\subset Y
\end{gather*}
The \textit{primal} problem can then be expressed as
\begin{equation}
    \label{eq: primal - min 2}
    \tag{$\mathfrak{P}$}
    \min\limits_{q\in Q}\left\{\,\min\limits_{\upalpha(\cdot)\in \mathcal{A}}\; \langle f(\cdot\,,\upalpha(\cdot),q),q\rangle,\quad \text{s.t.: } G(q)\in K_{\upalpha}\right\}
\end{equation}

\begin{lemma}\label{lem: existence primal}
    Let the assumptions (A), (B) and (C) be satisfied. Then the primal problem \eqref{eq: primal - min 2} has an optimal solution $(\mu_{\upalpha_{\circ}},\upalpha_{\circ})$. In particular, its value is finite. 
\end{lemma}

\begin{proof}[\textbf{Proof.}]
The proof is similar to  \cite[Lemma 4.1]{kouhkouh1}. We repeat it here for self-containedness. 
Recall that the feasible set of our \textit{primal} problem \eqref{eq: primal - min 2} is $\{q\in Q\,:\, G(q) \in K_{\upalpha}\} = \{\mu_{\upalpha}\}$ a singleton, where $\mu_{\upalpha}\in \mathcal{P}_{d}(\mathds{R}^{m})$. Hence, \eqref{eq: primal - min 2} is equivalent to
\begin{equation}
    \label{eq: primal sharp}
    \tag{$\mathfrak{P}_{_{_{\!\!\sharp}}}$}
    \min\limits_{\upalpha(\cdot)\in\mathcal{A}} \,F(\upalpha):=\langle f(\cdot\,,\upalpha(\cdot), \mu_{\upalpha}),\mu_{\upalpha} \rangle.
\end{equation}
the objective function $F(\cdot)$ is the one introduced in \eqref{functional F}. We are then minimizing a weak-$*$ continuous functional (thanks to Proposition \ref{prop: continuity F}) on the weak-$*$ compact\footnote{This is a consequence of Banach–Alaoglu's theorem; see e.g. \cite[Theorem 3.16, p.66]{brezis2011functional}} subset $\mathcal{A}$. Then $F(\cdot)$ is bounded on $\mathcal{A}$ and achieves its minimum on $\mathcal{A}$ (see \cite[Theorem 2, p.128]{luenberger1997optimization}). Its value is finite using assumption (A5(ii)) and existence of the minimum.
\end{proof}

As we have discussed in the companion paper \cite{kouhkouh1}, the latter existence result suggests a new description of the \textit{primal} problem \eqref{eq: primal - min 2}. Given an optimal solution $(\mu_{\upalpha_{\circ}},\upalpha_{\circ})\in X\times \mathcal{A}$, the problem \eqref{eq: primal - min 2} can be equivalently expressed as
\begin{equation}
    \label{eq: primal circ}
    \tag{$\mathfrak{P}_{_{_{\!\!\circ}}}$}
    \min\limits_{q\in Q} \,\langle f(\cdot\,,\upalpha_{\circ}(\cdot), \mu_{\upalpha_{\circ}}),q \rangle, \quad \text{ s.t.: } \quad G(q) \in K_{\upalpha_{\circ}}.
\end{equation}
Indeed, solving \eqref{eq: primal circ} yields the unique invariant probability measure $\mu_{\upalpha_{\circ}}$ since the feasible set is $\{q\in Q\,:\, G(q) \in K_{\upalpha_{\circ}}\} = \{\mu_{\upalpha_{\circ}}\}$. Yet, the advantage of using \eqref{eq: primal circ} is that, as we will later see, it is a convex problem for which strong duality holds. This is reminiscent of the \textit{hidden convexity} in the celebrated Benamou-Brenier formulation of optimal transport \cite{benamou2000computational}. See also Remark \ref{rmk: TV} for a comparison with optimal transport problem. 

To sum up, we have three equivalent formulations of the \textit{primal} problem: \eqref{eq: primal - min 2} will be used to construct the dual problem, \eqref{eq: primal sharp} is used to prove existence, and \eqref{eq: primal circ} to ensure strong duality holds. We refer to the end of \S 4.1 in \cite{kouhkouh1} for a more complete discussion. 


\subsection{The dual problem}

In order to deduce the corresponding \textit{dual} problem, we follow a parametric (conjugate) duality scheme as in \cite[\S 2.5.3, p. 107]{bonnans2013perturbation}. Therefore we embed the problem \eqref{eq: primal - min 2} in a family of parameterized problems, where $y\in Y$ is the parameter vector and consider the function (again using the notation in \S \ref{sec: primal prob})
\begin{equation*}
    \phi(q,y) = \min\limits_{\upalpha(\cdot)\in \mathcal{A}}\;\left\{\,  \langle f(\cdot\,,\upalpha(\cdot),q),q\rangle + I_{K_{\upalpha}}(G(q) + y)\,\right\}.
\end{equation*}
It is clear that when setting $y=0$, we recover the objective function in \eqref{eq: primal - min 2}. 

We also consider the following (Lagrangian) function, $L:X\times Y^{*}\times \mathcal{A} \to \mathds{R}$, analogue to \eqref{lagrangian} and such that
\begin{equation}\label{eq: Lagrangian - min}
    L(q,y^{*}, \upalpha) \coloneqq \langle f(\cdot\,,\upalpha(\cdot),q),q\rangle + \langle y^{*}, G(q) \rangle_{Y^{*},Y}.
\end{equation}
Using the Legendre-Fenchel transform, we have (see \cite[\S 4.2]{kouhkouh1})
\begin{equation*}
	\phi^{*}(q^{*},y^{*}) = \sup\limits_{q \in Q}\,\left\{ \langle q^{*},q \rangle - \min\limits_{\upalpha(\cdot)\in \mathcal{A}} \,\{L(q,y^{*},\upalpha) - I^{*}_{K_{\upalpha}}(y^{*})\}\,\right\}
\end{equation*}
The \textit{dual} of the parameterized \textit{primal} problem is then obtained as
\begin{equation*}
    \max\limits_{y^{*}\in Y^{*}}\,\{ \langle y^{*}, y \rangle - \phi^{*}(0,y^{*})\,\}
\end{equation*}
which is 
\begin{equation*}
    \max\limits_{y^{*}\in Y^{*}}\,\left\{ \langle y^{*}, y \rangle + \inf\limits_{q \in Q}\min\limits_{\upalpha(\cdot)\in \mathcal{A}} \{L(q,y^{*},\upalpha) - I^{*}_{K_{\upalpha}}(y^{*})\}\,\right\}
\end{equation*}
Finally, the \textit{dual} problem  to \eqref{eq: primal - min 2} is obtained by setting $y=0$, that is
\begin{equation}
    \label{eq: dual - min}
    \tag{$\mathfrak{D}$}
    \max\limits_{y^{*}\in Y^{*}}\,\left\{\inf\limits_{q \in Q}\min\limits_{\upalpha(\cdot)\in \mathcal{A}} \{L(q,y^{*},\upalpha) - I^{*}_{K_{\upalpha}}(y^{*})\}\,\right\}.
\end{equation}
We will now make \eqref{eq: dual - min} more explicit. 
\begin{lemma}\label{lem: main - min}
The problem \eqref{eq: dual - min} is equivalent to
\begin{equation}
    \label{eq: dual - min 1}
    \tag{$\mathfrak{D}$}
    \max\limits_{\substack{c\in\mathds{R}\\u\in \mathcal{X}}}\;\left\{\, c + \;\inf\limits_{q \in Q}\left\{\,\langle H(x,\nabla u, D^{2}u,q) - c , q \rangle\,\right\}\,\right\},
\end{equation}
where $H(x,\nabla u(x), D^{2}u(x),q) = \min\limits_{\alpha\in A}\{\, -\mathcal{L}_{\alpha}u(x) + f(x,\alpha,q) \,\}$ and $\mathcal{X}$ is such that
\begin{equation}\label{eq: functional space X - main}
    \mathcal{X} = D(\mathcal{L}_{0})\cap\{u:\mathds{R}^{m}\to \mathds{R}, \textit{Borel-meas.}\;|\; \exists\;C>0,\; |u(x)| \leq C(1+|x|^{\kappa})\}
\end{equation}
with $\kappa = d+1-\theta$,  that is, the two optimization problems have the same set of optimal solutions and the same optimal value.
\end{lemma}
\begin{remark}\label{rmk: domain of L in sobolev}
(A*) together with Theorem \ref{thm sobolev domain extension} ensure that $D(\mathcal{L}_{0}) \subset W^{r,2}_\text{loc}(\mathds{R}^{m})$.
\end{remark}

\begin{proof}[\textbf{Proof} of Lemma \ref{lem: main - min}]
The proof is in the line of the one of \cite[Lemma 3]{kouhkouh1}, the only difference being the dependence of $H$ (through $f$) on the measure $q$. We repeat it here for the sake of clarity. We have
\begin{equation}\label{eq: conjugate indicator}
    \begin{aligned}
    I^{*}_{K_{\alpha}}(y^{*}) = \sigma(y^{*};K_{\upalpha})
     = \left\{
    \begin{aligned}
    0 ,\quad   &\text{ if }\; y^{*}\in (K_{\upalpha})^{-}&\\
    +\infty,\quad   &\text{ otherwise }&
    \end{aligned}\right.
    \end{aligned}
\end{equation}
Recalling the definition $K_{\upalpha} = \{0\}\times \text{Ker}(\mathcal{L}_{\upalpha})$, we have
\begin{equation*}
    \begin{aligned}
    y^{*}\in (K_{\upalpha})^{-} & \Leftrightarrow\, (c,\omega) \in \bigg(\{0\}\times \text{Ker}(\mathcal{L}_{\upalpha})\bigg)^{-}\\
    & \Leftrightarrow\, (c,\omega) \in \mathds{R}\times (\text{Ker}(\mathcal{L}_{\upalpha}))^{\bot}\\
    & \Leftrightarrow\, (c,\omega) \in \mathds{R}\times \text{cl}(\text{range}(\mathcal{L}_{\upalpha}))
    \end{aligned}
\end{equation*}
Since we are working with $\mathcal{L}_{\upalpha}$ in its closed extension, we have
\begin{equation*}
    \begin{aligned}
    \omega\in \text{cl}(\text{range}(\mathcal{L}_{\upalpha})) & \Leftrightarrow\, \exists\;u\in D(\mathcal{L}_{\upalpha}),\;\text{s.t. }\; \omega = -\mathcal{L}_{\upalpha}u\\
    & \Leftrightarrow\, \exists\;u\in D(\mathcal{L}_{0}),\;\text{s.t. }\; \omega = -\mathcal{L}_{\upalpha}u
    \end{aligned}
\end{equation*}
where the last equivalence is obtained thanks to the assumption (A*) which guarantees that $D(\mathcal{L}_{\upalpha}) = D(\mathcal{L}_{0})$ for all $\upalpha(\cdot) \in \mathcal{A}$. The latter being independent of $\upalpha(\cdot)$, we can isolate it from the minimization over $\upalpha$ and write is as a subscript of the maximization over $(c,u)$. Then the \textit{dual} problem becomes 
\begin{equation}
    \label{eq: dual - min 3}
    \tag{$\mathfrak{D}$}
    \max\limits_{\substack{c\in\mathds{R}\\u\in D(\mathcal{L}_{0})}}\,\inf\limits_{q \in Q}\min\limits_{\upalpha(\cdot)\in \mathcal{A}} \{L(q,y^{*},\upalpha),\;\;\text{ s.t.: }\, y^{*}=(c,-\mathcal{L}_{\upalpha}u)\,\}.
\end{equation}
Recalling the definition \eqref{eq: Lagrangian - min} of $L$ and the notations introduced earlier, we have
\begin{equation*}
    \begin{aligned}
    L(q, y^{*},\upalpha) & = \langle f(\cdot\,,\upalpha(\cdot),q),q\rangle + \langle y^{*}, G(q) \rangle_{Y^{*},Y}\\
    & = \langle f(\cdot\,,\upalpha(\cdot),q) ,q\rangle + c(1-\langle 1, q \rangle) + \langle -\mathcal{L}_{\upalpha}u(\cdot),q \rangle\\
    & = c + \langle f(\cdot\,,\upalpha(\cdot),q) - \mathcal{L}_{\upalpha}u(\cdot) - c , q \rangle
    \end{aligned}
\end{equation*}
hence we have, using the \textit{exchange property} in Proposition \ref{prop: exchange prop},
\begin{equation*}
    \begin{aligned}
    & \min\limits_{\upalpha(\cdot)\in \mathcal{A}} \{L(q,y^{*},\upalpha),\;\;\text{ s.t.: }\, y^{*}=(c,-\mathcal{L}_{\upalpha}u)\,\} \\
    & \quad \quad \quad \quad \quad = c + \min\limits_{\upalpha(\cdot)\in \mathcal{A}}\,\left\{\, \langle f(\cdot\,,\upalpha(\cdot),q) - \mathcal{L}_{\upalpha}u(\cdot) - c , q \rangle\,\right\}\\
    & \quad \quad \quad \quad \quad = c + \langle \min\limits_{\alpha\in A}\{f(\cdot\,,\alpha,q) - \mathcal{L}_{\alpha}u(\cdot)\} - c , q \rangle\\
    & \quad \quad \quad \quad \quad = c + \langle H(x,\nabla u, D^{2}u,q) - c , q \rangle.
    \end{aligned}
\end{equation*}
But since $Q$ is made of non-negative measures with finite moment of order $d$, we need $u$ to have a polynomial growth of order at most $\kappa=d+1-\theta$ (see (A3), (A5) and (A6)). The \textit{dual} problem finally takes the form
\begin{equation*}
    \max\limits_{\substack{c\in\mathds{R}\\u\in \mathcal{X}}}\;\left\{\, c + \;\inf\limits_{q \in Q}\langle H(x,\nabla u, D^{2}u,q) - c , q \rangle\,\right\}.
\end{equation*}
where the functional space $\mathcal{X}$ is defined as
\begin{equation*}
    \mathcal{X} = D(\mathcal{L}_{0})\cap\{u:\mathds{R}^{m}\to \mathds{R}, \textit{Borel-meas.}\;|\; \exists\;C>0,\; |u(x)| \leq C(1+|x|^{\kappa})\}
\end{equation*}
and $\kappa=d+1-\theta$. This concludes the proof.
\end{proof}

\section{Main result: ergodic MFG system}\label{sec: main results}

\subsection*{The PDE problem}

We address the problem of existence of solutions to an ergodic mean-field games (MFG) system, that is
\begin{equation}
    \label{eq: mfg - main}
    \begin{aligned}
    &\quad\quad\textit{Find } (c,u, \mu)\in\mathds{R}\times \mathcal{X}(\mathds{R}^{m})\times\mathcal{P}(\mathds{R}^{m}),\, \textit{s.t.:}\\ 
            & H(x,\nabla u(x),D^{2}u(x),\mu) = c\quad  \text{ and }\; - \mathcal{L}^{*}_{\upalpha_{[u,\mu]}}\mu = 0 
    \end{aligned}
\end{equation}
where $\mathcal{X}$ is a functional space (part of the unknowns), $\mathcal{P}$ is the set of probability measures and the Hamiltonian is of the form 
\begin{equation}\label{eq: Ham}
     H(x,\nabla u(x),D^{2}u(x),\mu)\coloneqq \min\limits_{\alpha\in A}\{\,-\mathcal{L}_{\alpha}u(x) + f(x,\alpha,\mu)\,\},
\end{equation}
the diffusion operator $\mathcal{L}_{\alpha}$ is a linear operator given by
\begin{equation*}
    \mathcal{L}_{\alpha}\varphi(x) \coloneqq \text{trace}\big( a(x,\alpha)D^{2}\varphi(x)\big) + b(x,\alpha)\cdot\nabla\varphi(x) 
\end{equation*}
and its adjoint $\mathcal{L}^{*}_{\alpha}$ is then
\begin{equation*}
    \mathcal{L}^{*}_{\alpha}\rho(x) = \text{trace}\big(D^{2}(a(x,\alpha)\rho(x))\big) - \text{div}\big(b(x,\alpha)\rho(x)\big).
\end{equation*}
The second equation in \eqref{eq: mfg - main} is nothing but $-\mathcal{L}^{*}_{\upalpha}\mu = 0$ where $\upalpha \equiv \upalpha_{[u,\mu]}(\cdot)\in \mathcal{A}$ is a function of $x$ and it depends on $u$ and $\mu$ such that
\begin{equation*}
	\upalpha_{[u,\mu]}(x) \in \argmin\limits_{\alpha \in A}\{\,-\mathcal{L}_{\alpha}u(x) + f(x,\alpha,\mu)\,\}.
\end{equation*}
The case where $H$ is given with a $\max$ (instead of a $\min$) can be obtained analogously (see \cite{kouhkouhPhD} for further details).

\subsection*{The optimality conditions}

We check that the optimality conditions as stated in \S\ref{sec: duality}, in particular \eqref{optimality conditions - 2} and \eqref{equiv cond normal cone}, still hold in our framework. In order to do so, we start from the \textit{duality gap} (or \textit{duality inequality}) which states that the value of the \textit{dual} problem \eqref{eq: dual - min} is less or equal than the value of the \textit{primal} problem \eqref{eq:  primal - min 2}. Recalling the definition \eqref{eq: Lagrangian - min} of the Lagrangian function $L$ 
\begin{equation*}
    L(q,y^{*}, \upalpha) = \langle f(\cdot\,,\upalpha(\cdot),q),q\rangle + \langle y^{*}, G(q) \rangle_{Y^{*},Y}
\end{equation*}
and the value of the \textit{dual} problem being less or equal the value of the \textit{primal} problem (see \S \ref{sec: duality}), we have
\begin{equation*}
\begin{aligned}
    &\max\limits_{y^{*}\in Y^{*}}\min\limits_{q\in Q}\min\limits_{\upalpha(\cdot)\in \mathcal{A}}\{L(q,y^{*},\upalpha)-I_{K_{\upalpha}}^{*}(y^{*})\}\\
    &\quad\quad\quad \leq \min\limits_{q\in Q}\min\limits_{\upalpha(\cdot)\in \mathcal{A}}\{\langle f(\cdot\,,\upalpha(\cdot),q),q\rangle + I_{K_{\upalpha}}(G(q))\}\\
    &\quad\quad\quad \leq \min\limits_{q\in Q}\min\limits_{\upalpha(\cdot)\in \mathcal{A}}\{ L(q,y^{*},\upalpha) + I_{K_{\upalpha}}(G(q)) - \langle y^{*},G(q) \rangle_{Y^{*},Y}\},\;\forall\,y^{*}\in Y^{*}.
\end{aligned}
\end{equation*}
Let us denote by $(q_{\circ},\upalpha_{\circ})$ an optimal solution in the \textit{primal} problem \eqref{eq: primal - min 2} and by $y^{*}_{\circ}$ an optimal solution in the \textit{dual} problem \eqref{eq: dual - min}. We then have
\begin{equation}\label{eq: weak duality}
\begin{aligned}
    & \min\limits_{q\in Q}\min\limits_{\upalpha(\cdot)\in \mathcal{A}}\{L(q,y^{*}_{\circ},\upalpha)-I_{K_{\upalpha}}^{*}(y^{*}_{\circ})\} \\
    &\quad \quad \quad \quad \quad \quad  \leq  L(q_{\circ},y^{*}_{\circ},\upalpha_{\circ}) + I_{K_{\upalpha_{\circ}}}(G(q_{\circ})) - \langle y^{*}_{\circ},G(q_{\circ}) \rangle_{Y^{*},Y}\\
    &\quad \quad \quad \quad \quad \quad \quad \quad \quad \quad \quad \quad  = \langle f(\cdot\,, \upalpha_{\circ}(\cdot),q_{\circ}), q_{\circ} \rangle + I_{K_{\upalpha_{\circ}}}(G(q_{\circ})).
\end{aligned}
\end{equation}

The optimality conditions are obtained when we reach equality in the above inequality. We can then characterize the optimal \textit{primal} and \textit{dual} solutions and provide a no-\textit{duality gap} condition. Suppose the left hand side minimization in the above inequality is reached in the pair of optimal solutions $(q_{\circ},\upalpha_{\circ})$. Therefore, the latter inequality reduces to 
\begin{equation*}
    0 \leq I^{*}_{K_{\upalpha_{\circ}}}(y^{*}_{\circ}) + I_{K_{\upalpha_{\circ}}}(G(q_{\circ})) - \langle y^{*}_{\circ}, G(q_{\circ}) \rangle_{Y^{*},Y}.
\end{equation*}
This is the Fenchel-Young inequality, and equality holds if and only if we have
\begin{equation}\label{eq: opt cond main - 3}
    y^{*}_{\circ}\in \partial I_{K_{\upalpha_{\circ}}}(G(q_{\circ})) = N_{K_{\upalpha_{\circ}}}(G(q_{\circ})).
\end{equation}
Since $K_{\upalpha_{\circ}}$ is a convex cone, then $y^{*}_{\circ}\in N_{K_{\upalpha_{\circ}}}(G(q_{\circ}))$ is equivalent to 
\begin{equation}
    G(q_{\circ})\in K_{\upalpha_{\circ}},\quad y^{*}_{\circ}\in (K_{\upalpha_{\circ}})^{-}\; \text{ and }\; \langle y^{*}_{\circ},G(q_{\circ}) \rangle_{Y^{*},Y} = 0.
\end{equation}
Recalling the definition \eqref{eq: conjugate indicator}, we have $I^{*}_{K_{\upalpha_{\circ}}}(y^{*}_{\circ}) = 0$ when $y^{*}_{\circ}\in (K_{\upalpha_{\circ}})^{-}$. So going back to the inequality in \eqref{eq: weak duality}, which we are now supposing to be an equality (\textit{no-duality gap}), we have 
\begin{equation*}
\begin{aligned}
    \min\limits_{q\in Q}\min\limits_{\upalpha(\cdot)\in \mathcal{A}}\{L(q,y^{*}_{\circ},\upalpha)-I_{K_{\upalpha}}^{*}(y^{*}_{\circ})\} =  L(q_{\circ},y^{*}_{\circ},\upalpha_{\circ}) - I_{K_{\upalpha_{\circ}}}^{*}(y^{*}_{\circ}) =  L(q_{\circ},y^{*}_{\circ},\upalpha_{\circ}).
\end{aligned}
\end{equation*}
Recalling \eqref{eq: conjugate indicator}, we have $\min\limits_{q\in Q}\min\limits_{\upalpha(\cdot)\in \mathcal{A}}\{L(q,y^{*}_{\circ},\upalpha)-I_{K_{\upalpha}}^{*}(y^{*}_{\circ})\} \leq  \min\limits_{q\in Q}\min\limits_{\upalpha(\cdot)\in \mathcal{A}} \; L(q,y^{*}_{\circ},\upalpha)$ 
which finally yields, together with the previous equality,
\begin{equation*}
    L(q_{\circ},y^{*}_{\circ},\upalpha_{\circ}) \leq \min\limits_{q\in Q}\min\limits_{\upalpha(\cdot)\in \mathcal{A}} \; L(q,y^{*}_{\circ},\upalpha).
\end{equation*}

To sum up, we have the following sufficient optimality conditions which also guarantee the absence of the \textit{duality gap}
\begin{equation} \label{optimality conditions - main}
\left\{\;
    \begin{aligned}
    & (q_{\circ},\upalpha_{\circ}) \in \argmin\limits_{q\in Q, \upalpha(\cdot)\in \mathcal{A}}\; L(q,y^{*}_{\circ},\upalpha)\\
    & G(q_{\circ})\in K_{\upalpha_{\circ}},\quad y^{*}_{\circ}\in (K_{\upalpha_{\circ}})^{-}\; \text{ and }\; \langle y^{*}_{\circ},G(q_{\circ}) \rangle_{Y^{*},Y} = 0.
    \end{aligned}
\right.
\end{equation}
They are indeed analogue to \eqref{optimality conditions - 2}. 

\subsection{Existence and uniqueness}

Our  main result is a necessary and sufficient theorem for existence and uniqueness of a solution to  ergodic MFG system \eqref{eq: mfg - main}.

\begin{theorem}\label{thm: nec and suf - min}
Assuming (A), (B), (C) and (A*) hold true, the following statements are equivalent
\begin{enumerate}[label = (\Roman*)]
    \item The primal problem \eqref{eq: primal - min} admits a solution $(q_{\circ},\upalpha_{\circ})$, that is,
    \begin{equation*}
         (q_{\circ},\upalpha_{\circ}) \in \argmin\limits_{\substack{q\in\mathcal{M}_{d}^{+}(\mathds{R}^{m})\\ \upalpha(\cdot)\in\mathcal{A}}}\big\{\langle f(\cdot\,,\upalpha(\cdot),q),\,q\,\rangle\, ,\;\text{s.t.: } 1-\langle 1,q\rangle = 0 \text{ and } q\in \text{Ker}(\mathcal{L}^{*}_{\upalpha})\big\}.
    \end{equation*}
    \item{
        There exist $(c_{\circ},u_{\circ},q_{\circ})\in\mathds{R}\times W^{r,2}_{\text{loc}}(\mathds{R}^{m})\times W^{s,1}_{\text{loc}}(\mathds{R}^{m})$ for any $r\geq 1$, $s>m$ and  a measurable function $\upalpha_{\circ}(\cdot):\mathds{R}^{m}\to A$, 
        solving the MFG system
        \begin{equation}\label{eq: pde system - thm}
        \left\{
            \begin{aligned}
            & \quad \min\limits_{\alpha\in A}\{ -\text{trace}\big( a(x,\alpha)D^{2} u_{\circ}(x)\big) - b(x,\alpha)\cdot\nabla u_{\circ}(x) +  f(x,\alpha,q_{\circ}) \} = c_{\circ} \\
            & - \text{trace}\big(D^{2}(a(x,\upalpha_{\circ}(x))q_{\circ}(x))\big) + \text{div}\big(b(x,\upalpha_{\circ})q_{\circ}(x)\big) = 0, \quad \quad  \text{a.e. in } \mathds{R}^{m}
            \end{aligned}
        \right.
        \end{equation}
        and moreover
        \begin{enumerate}[label = (\alph*)]
            \item the constant $c_{\circ}$ is defined by $c_{\circ}=\langle f(\cdot\,,\upalpha_{\circ}(\cdot),q_{\circ}),q_{\circ} \rangle$,
            \item $u_{\circ}(\cdot)$ satisfies: $|u_{\circ}(x)|\leq K(1+|x|^{\kappa})$, with $\kappa = d+1-\theta$ and $K>0$ a constant,
            \item $q_{\circ}(\cdot)$ is the density of a probability measure, absolutely continuous w.r.t. Lebesgue,
            \item $\upalpha_{\circ}(\cdot)$ satisfies $\upalpha_{\circ}(x)\in \argmin\limits_{\alpha\in A}\{-\mathcal{L}_{\alpha}u_{\circ}(x) + f(x,\alpha,q_{\circ})\}$ a.e. $x\in\mathds{R}^{m}$.
        \end{enumerate}
    }
\end{enumerate}
If in addition $(q_{\circ},\upalpha_{\circ})$ in (I) is unique and  the vector field $b$ is locally Lipschitz continuous in $x$ with $\theta=1$ in (A6), then $u_{\circ}(\cdot)$ is unique in $W^{r,2}_{loc}(\mathds{R}^{m})$ with $r>\frac{m}{2}$ for $c_{\circ}$ given in (II-a), that is, if $(c_{\circ},u_{1}(\cdot))$ and $(c_{\circ},u_{2}(\cdot))$ are two solutions as in (II), then $u_{1}(\cdot)-u_{2}(\cdot)$ is a constant.
\end{theorem}

\begin{remark}\label{rmk: regularity}  In fact $u_{\circ}(\cdot)$ is an $L$-viscosity solution (see e.g. \cite{caffarelli1996viscosity, crandall1996equivalence}), which is as expected as when we consider $C$-viscosity solutions for the continuous  case. Recall in our setting, the vector field $b$ and the function $f$ are merely measurable  in $x$.
\end{remark}

\begin{remark}\label{rem: uniq} Some observations on uniqueness of the solution:\\
\textbullet\quad Uniqueness of $(q_{\circ},\upalpha_{\circ})$ in statement (I) requires the (primal) optimization problem to be jointly convex in $(q,\alpha)$. This is hardly satisfied because of the constraint $q\in Ker(\mathcal{L}^{*}_{\alpha})$. Therefore, one does not expect uniqueness for the MFG system.\\
\textbullet\quad The constant $c_{\circ}$ is in general not unique. In fact, there might be infinitely many constants for which there exists a solution $(u,q)$. See \cite[Remark 4.7]{kouhkouh1} and \cite{ichihara2011recurrence, kaise2006structure}.
\end{remark}

Note that by the latter theorem, we reduced the problem of existence of a solution $(c,u,q)$ for the MFG system \eqref{eq: mfg - main} to the solvability of an (infinite dimensional) optimization problem where the unknown is $(q,\upalpha)$. 

A direct consequence of Theorem \ref{thm: nec and suf - min} and Lemma \ref{lem: existence primal} is the following.

\begin{corollary}\label{cor: main 2}
    The ergodic problem \eqref{eq: mfg - main}  (equivalently, the MFG system \eqref{eq: pde system - thm}) admits a solution as in the statement (II) of Theorem \ref{thm: nec and suf - min}.
\end{corollary}


\begin{proof}[\textbf{Proof} of Theorem \ref{thm: nec and suf - min}]
The proof is a consequence of Theorem \ref{thm: equivalence} and Lemma \ref{lem: existence primal}, provided we express the optimality conditions \eqref{optimality conditions - main} in terms of a PDE system as in the statement (II). And to do so, we rely on Lemma \ref{lem: main - min} and on the results in \S\ref{sec: opt space measure}. 

But before we go any further, let $(q_{\circ},\upalpha_{\circ})=(\mu_{\upalpha_{\circ}},\upalpha_{\circ})$ be an optimal solution for \eqref{eq: primal sharp} as given by Lemma \ref{lem: existence primal}, and let us consider the \textit{primal} problem in its formulation
\begin{equation}
    \label{eq: primal circ proof}
    \tag{$\mathfrak{P}_{_{_{\!\!\circ}}}$}
    \min\limits_{q\in Q} \,\langle f(\cdot\,,\upalpha_{\circ}(\cdot), \mu_{\upalpha_{\circ}}),q \rangle, \quad \text{ s.t.: } \quad G(q) \in K_{\upalpha_{\circ}}.
\end{equation}

\textit{Step 1.  (On the optimization problems)}\\
We need to check if the assumptions of Theorem \ref{thm: equivalence} are satisfied by \eqref{eq: primal circ proof}. The objective function $q\mapsto \langle f(\cdot\,\upalpha_{\circ}(\cdot),\mu_{\upalpha_{\circ}}),q \rangle$ is linear hence convex and continuous, the set $Q=\mathcal{M}^{+}_{d}(\mathds{R}^{m})$ is clearly convex and close, the function $G(q) = (G_{1}(q),G_{2}(q))$, with $G_{1}(q) = 1-\langle 1,q\rangle$ and $G_{2}(q)=q$, is continuously differentiable and convex w.r.t. the set $-K$ (this is easy to check as $G$ is affine). The last assumption we need is \eqref{eq: interior} which is in our situation equivalent to \eqref{eq: interior 2} as shown by Proposition \ref{prop: interior}. Let $q$ be a feasible point and recall the notation in \S \ref{sec: primal prob}. Using the results in \S \ref{sec: diffusion}, in particular (\ref{thm existence inv meas}) and (\ref{thm regularity meas}) in Theorem \ref{thm summary diff}, we have $K_{2}(\upalpha_{\circ})= \text{Ker}(\mathcal{L}^{*}_{\upalpha_{\circ}}) = \{h\,:\, h= \lambda \mu_{\upalpha_{\circ}},\,\lambda \geq 0\}$. We can then write 
\begin{equation*}
\begin{aligned}
    & G_{1}(q) +DG_{1}(q)[K_{2} - q] - K_{1}\\
    & \quad \quad \quad = 1- \langle 1, q\rangle + \{\,-\langle 1, h - q \rangle \;:\; \forall\, h\in K_{2}(\upalpha_{\circ})\}\\
    & \quad \quad \quad = 1 - \{\,\lambda\,\langle 1, \mu_{\upalpha_{\circ}}  \rangle \;:\; \forall\,\lambda\geq 0\} = (-\infty,1]
\end{aligned}
\end{equation*}
where in the last equality we used the fact that $\mu_{\upalpha_{\circ}}$ is a probability measure hence $\langle 1,\mu_{\upalpha_{\circ}}\rangle = 1$. Therefore $0\in \text{int}\{G_{1}(q) +DG_{1}(q)[K_{2} - q] - K_{1}\}$ and we can indeed apply Theorem \ref{thm: equivalence} since Lemma \ref{lem: existence primal} ensures that the primal problem has a solution and hence a finite value. In particular: 
\begin{itemize}
    \item the first part of Theorem \ref{thm: equivalence} ensures $(i)$ no duality gap between the primal and dual problems and, $(ii)$ the primal problem admits a solution if and only there exists an element in the dual space $Y^{*}$ satisfying the optimality conditions \eqref{optimality conditions - main}, then
    \item the second part of Theorem \ref{thm: equivalence} together with Lemma \ref{lem: existence primal} ensure that the elements of the dual space $Y^{*}$ satisfying the optimality conditions \eqref{optimality conditions - main} are optimal for the dual problem. 
\end{itemize}
Next, we need to translate these optimality conditions into a PDE. 

\textit{Step 2. (On the optimality conditions \eqref{optimality conditions - main})}\\
Let us denote by $y^{*}_{\circ}\in Y^{*}$  an optimal solution of the \textit{dual} problem \eqref{eq: dual - min}. Following Lemma \ref{lem: main - min} (see also its proof), one can substitute the dual variables $y^{*}$ with the pairs of variables $(c,u)\in\mathds{R}\times \mathcal{X}$ where $\mathcal{X}$ is as defined in \eqref{eq: functional space X - main}. And the optimal dual variables are given by $y^{*}_{\circ}= (c_{\circ}, - \mathcal{L}_{\upalpha_{\circ}}u_{\circ})$.

Now, the no-\textit{duality gap} yields
\begin{equation}\label{eq: no duality gap - proof}
    c_{\circ} + \;\inf\limits_{q \in Q}\left\{\,\langle H(x,\nabla u_{\circ}, D^{2}u_{\circ},q) - c_{\circ} , q \rangle\,\right\} = \langle f(\cdot\,, \upalpha_{\circ}(\cdot),q_{\circ}), q_{\circ} \rangle
\end{equation}
and the last condition in \eqref{optimality conditions - main} that is $\langle y_{\circ}^{*},G(q_{\circ}) \rangle_{Y^{*},Y} = 0$, can be expressed as 
\begin{equation*}
    c_{\circ}(1-\langle 1, q_{\circ}\rangle) + \langle -\mathcal{L}_{\upalpha_{\circ}}u_{\circ}(\cdot) , q_{\circ} \rangle = 0,
\end{equation*}
i.e. $c_{\circ} = \langle c_{\circ}, q_{\circ}\rangle - \langle -\mathcal{L}_{\upalpha_{\circ}}u_{\circ}(\cdot) , q_{\circ} \rangle$. Substituting $c_{\circ}$ in \eqref{eq: no duality gap - proof} yields
\begin{equation}\label{eq: 1 - proof}
    \inf\limits_{q \in Q}\,\langle H(x,\nabla u_{\circ}, D^{2}u_{\circ},q) - c_{\circ} , q \rangle\, = \langle - \mathcal{L}_{\upalpha_{\circ}}u_{\circ}(\cdot) + f(\cdot\,,\upalpha_{\circ}(\cdot),q_{\circ}) - c_{\circ} , q_{\circ}\rangle.
\end{equation}
Thanks to the \textit{exchange property} \eqref{eq: exchange prop} and recalling the definition of the Hamiltonian \eqref{eq: Ham}, the latter equality becomes
\begin{equation}\label{eq: 2 - proof}
    \inf\limits_{q \in Q}\min\limits_{\upalpha(\cdot)\in\mathcal{A}}\,\langle -\mathcal{L}_{\upalpha}u_{\circ}(\cdot) + f(\cdot\,,\upalpha(\cdot), q) - c_{\circ} , q \rangle\, = \langle - \mathcal{L}_{\upalpha_{\circ}}u_{\circ}(\cdot) + f(\cdot\,,\upalpha_{\circ}(\cdot),q_{\circ}) - c_{\circ} , q_{\circ}\rangle,
\end{equation}
that is
\begin{equation}\label{eq: min problem measure - control - proof}
    (q_{\circ},\upalpha_{\circ}) \in \argmin\limits_{\substack{q\in Q\\ \upalpha(\cdot)\in\mathcal{A}}}\,\big\{\;\langle -\mathcal{L}_{\upalpha}u_{\circ}(\cdot) + f(\cdot\,,\upalpha(\cdot), q) - c_{\circ} , q \rangle\;\big\}.
\end{equation}
In particular, when setting $q$ to its optimal value $q_{\circ}$, one has
\begin{equation}
    \upalpha_{\circ}(\cdot) \in \argmin\limits_{\upalpha(\cdot)\in \mathcal{A}}\,\big\{\;\langle -\mathcal{L}_{\upalpha}u_{\circ}(\cdot) + f(\cdot\,,\upalpha(\cdot), q_{\circ}) - c_{\circ} , q_{\circ} \rangle\;\big\}
\end{equation}
which yields thanks to the \textit{exchange property} \eqref{eq: exchange prop}
\begin{equation}\label{eq: charact upalpha - proof}
    \upalpha_{\circ}(x) \in \argmin\limits_{\alpha \in A}\,\big\{\, -\mathcal{L}_{\alpha}u_{\circ}(x) + f(x,\alpha, q_{\circ})\;\big\},\quad q_{\circ}-\text{a.e. } x\in \mathds{R}^{m},
\end{equation}
i.e. $H(x,\nabla u_{\circ}(x),D^{2}u_{\circ}(x),q_{\circ}) = -\mathcal{L}_{\upalpha_{\circ}}u_{\circ}(x) + f(x,\upalpha_{\circ}(x), q_{\circ})$, $q_{\circ}$-almost everywhere. And thanks to (\ref{thm regularity meas}) in Theorem \ref{thm summary diff}, $q_{\circ}$ is absolutely continuous with respect to Lebesgue measure and hence the result almost everywhere in $\mathds{R}^{m}$.

Analogously, when setting $\upalpha(\cdot)$ to its optimal value $\upalpha_{\circ}(\cdot)$ in \eqref{eq: min problem measure - control - proof}, one has 
\begin{equation}\label{eq: min problem measure - proof}
    q_{\circ} \in \argmin\limits_{q\in Q}\,\big\{\;\langle -\mathcal{L}_{\upalpha_{\circ}}u_{\circ}(\cdot) + f(\cdot\,,\upalpha_{\circ}(\cdot), q) - c_{\circ} , q \rangle\;\big\}.
\end{equation}
And recalling the definition of the \textit{primal} problem \eqref{eq: primal - min}, the condition $G(q_{\circ})\in K_{\upalpha_{\circ}}$ in  \eqref{optimality conditions - main} means in particular that $\langle 1,q_{\circ} \rangle=1$, and since $q\in Q = \mathcal{M}_{d}^{+}(\mathds{R}^{m})$, then $q_{\circ}$ is a probability measure.

We will now show (using the results in \S \ref{sec: opt space measure}) that an optimality condition for the optimization problem \eqref{eq: min problem measure - proof} allows to prove that $(c_{\circ},u_{\circ})$ solves the PDE $- \mathcal{L}_{\upalpha_{\circ}}u_{\circ} + f(\cdot\,,\upalpha_{\circ}(\cdot),q_{\circ}) = c_{\circ}$ a.e. in $\mathds{R}^{m}$, i.e. $H(x,\nabla u_{\circ}(x),D^{2}u_{\circ}(x),q_{\circ}) = c_{\circ}$ a.e. in $\mathds{R}^{m}$.

\textit{Step 2.1. (On the problem \eqref{eq: min problem measure - proof})}\\
We define $\widetilde{f}:\mathds{R}^{m}\times \mathcal{M}(\mathds{R}^{m})\to \mathds{R}\,$ and $g:\mathds{R}^{m}\to \mathds{R}$ respectively by
\begin{equation*}
    \widetilde{f}(x,q)\coloneqq f(x,\upalpha_{\circ}(x),q),\quad \quad g(x) \coloneqq  - \mathcal{L}_{\upalpha_{\circ}}u_{\circ}(x) -c_{\circ}
\end{equation*}
and we set
\begin{equation*}
    \Psi(q) \coloneqq \langle\; \widetilde{f}(\,\cdot\,,q) + g(\cdot) \,,\, q\;\rangle.
\end{equation*}
The optimization problem \eqref{eq: min problem measure - proof} writes equivalently as
\begin{equation}\label{eq: min problem measure - proof 2}
    \min\;\big\{\; \Psi(q)\,,\quad \text{s.t.: }\; q\in \mathcal{M}^{+}_{d}(\mathds{R}^{m})\;\big\}.
\end{equation}
With this formulation, it is easy to see that any measure $q$ satisfying the constraint in \eqref{eq: min problem measure - proof 2} is \textit{regular} in the sense of Definition \ref{def: regular measure}. Indeed, it suffices to set, in the notation of \eqref{eq: regular measure}, $Q=\mathcal{M}_{d}^{+}(\mathds{R}^{m})$, $G(q)=q$ and $K=\mathcal{M}^{+}(\mathds{R}^{m})$. Thanks to assumption (C1), the function $\Psi$ is Fréchet differentiable and we can apply Theorem \ref{thm: nec cond measure opt} together with Theorem \ref{thm: tangent set measure} and Corollary \ref{cor: tangent set measure} to obtain the following first-order necessary condition for $q_{\circ}$ to be a minimum of \eqref{eq: min problem measure - proof 2} (or equivalently of \eqref{eq: min problem measure - proof}): 
\begin{equation}\label{eq: nec cond measure - proof}
    D\Psi(q_{\circ})[h] \geq 0,\quad \forall\, h\in\{\,h\in\mathcal{M}_{d}(\mathds{R}^{m})\,:\, h^{-}\ll q_{\circ}\},
\end{equation}
where, using the definition of $\Psi$, one has
\begin{equation*}
        D\Psi(q_{\circ})[h] = \langle \widetilde{f}(\cdot\,,q_{\circ}) + g(\cdot)\,,\, h \rangle + \langle D_{\mu}\widetilde{f}(\cdot\,,\,q_{\circ})[h], q_{\circ} \rangle.
\end{equation*}

\textit{Step 2.2. (We show that $\widetilde{f}(\cdot\,,q_{\circ}) + g(\cdot)\geq 0$ in $\mathds{R}^{m}$)}\\
We proceed by contradiction. Suppose $\exists\,\overline{x}\in\mathds{R}^{m}$ such that $\widetilde{f}(\overline{x},q_{\circ}) + g(\overline{x}) <0$.\\
We choose $h=\delta_{\overline{x}}$, the Dirac measure with unit mass concentrated at $\overline{x}$. It is a positive measure and is clearly in $T_{\mathcal{M}^{+}_{d}(\mathds{R}^{m})}(q_{\circ})$. When used in \eqref{eq: nec cond measure - proof}, one gets
\begin{equation*}
    \begin{aligned}
        0 & \;\leq\quad  \langle \widetilde{f}(\cdot\,,q_{\circ}) + g(\cdot)\,,\, \delta_{\overline{x}} \rangle + \langle D_{\mu}\widetilde{f}(\cdot\,,\,q_{\circ})[\delta_{\overline{x}}], q_{\circ} \rangle\\
        & \;\leq\quad  \widetilde{f}(\overline{x},q_{\circ}) + g(\overline{x}) + \langle D_{\mu}\widetilde{f}(\cdot\,,\,q_{\circ})[\delta_{\overline{x}}], q_{\circ} \rangle
    \end{aligned}
\end{equation*}
But using assumption (C2), we have $\langle D_{\mu}\widetilde{f}(\cdot\,,\,q_{\circ})[\delta_{\overline{x}}], q_{\circ} \rangle\leq 0$ and this yields a contradiction with $\widetilde{f}(\overline{x},q_{\circ}) + g(\overline{x})<0$. Hence, the function $\widetilde{f}(\cdot\,,q_{\circ}) + g(\cdot)$ is non-negative for all $x\in\mathds{R}^{m}$.

\textit{Step 2.3. (We show that $\widetilde{f}(x,q_{\circ}) + g(x)= 0$ almost everywhere in $\mathds{R}^{m}$)}\\
We proceed by contradiction. Suppose there exists a Borel subset $B$ (open set in $\mathds{R}^{m}$) such that $q_{\circ}(B)\neq 0$ and a constant $\Gamma>0$, such that
\begin{equation*}
    \Gamma := q_{\circ}-\text{ess}\sup\limits_{x\in B}\{\; \widetilde{f}(x,q_{\circ}) + g(x) \;\} = \inf\{\gamma\in\mathds{R}\,: \; \widetilde{f}(x,q_{\circ}) + g(x) \leq \gamma,\; q_{\circ}-\text{a.e. in } B\}.
\end{equation*}

We will first show that the pair $(q_{\circ},\upalpha_{\circ})$ in the problem \eqref{eq: min problem measure - control - proof} remains the same when we subtract to $f(\cdot\,,\upalpha(\cdot),q)$ a positive constant. Then we will show that $\Gamma$ cannot be positive, which together with the previous \textit{Step 2.2} yields the desired result. 

Observe that $(q_{\circ},\upalpha_{\circ})$ besides being a minimizer for the problem \eqref{eq: min problem measure - control - proof}, it is determined by the optimality conditions \eqref{optimality conditions - main}. In particular, it is a minimizer for the \textit{primal} problem \eqref{eq: primal - min}. Therefore, we start from the latter problem \eqref{eq: primal - min} where we will subtract to $f$ a constant $n\Gamma$ where $n\geq 1$ (although the choice of the constant here is not important, we keep considering $\Gamma$ as defined above to avoid introducing new constants). \\
Recall the \textit{primal} problem formulated as 
\begin{equation}
    \label{eq: primal sharp proof}
    \tag{$\mathfrak{P}_{_{_{\!\!\sharp}}}$}
    \min\limits_{\upalpha(\cdot)\in\mathcal{A}} \,\langle f(\cdot\,,\upalpha(\cdot), \mu_{\upalpha}),\mu_{\upalpha} \rangle.
\end{equation}
Subtracting a constant $n\Gamma$ to $f$ in \eqref{eq: primal sharp proof} yields the optimization problem
\begin{equation*}
    \min\limits_{\upalpha(\cdot)\in\mathcal{A}}\; \langle f(\cdot\,,\upalpha(\cdot),\mu_{\upalpha})-n\Gamma\,,\,\mu_{\upalpha} \rangle.
\end{equation*}
But $\mu_{\upalpha}$ being a probability measure, the latter can be written as
\begin{equation*}
    -n\Gamma \,+\,\min\limits_{\upalpha(\cdot)\in\mathcal{A}}\; \langle f(\cdot\,,\upalpha(\cdot),\mu_{\upalpha})\,,\,\mu_{\upalpha} \rangle.
\end{equation*}
And $(q_{\circ},\upalpha_{\circ}) = (\mu_{\upalpha_{\circ}},\upalpha_{\circ})$ is again a minimizer for the latter problem. In other words, subtracting a constant to $f$ in the objective function in \eqref{eq: primal - min} does not alter the optimality of the pair $(q_{\circ},\upalpha_{\circ})$. And ultimately the optimality conditions \eqref{optimality conditions - main} also remain the same. 

Therefore, one can still consider $(c_{\circ},u_{\circ},q_{\circ},\upalpha_{\circ})$ as in \eqref{eq: 2 - proof} even if we subtract to $f$ a constant $n\Gamma$, i.e.
\begin{equation*}
\begin{aligned}
   & \inf\limits_{q \in Q}\min\limits_{\upalpha(\cdot)\in\mathcal{A}}\left\{\,\langle -\mathcal{L}_{\upalpha}u_{\circ} + f(\cdot\,,\upalpha(\cdot), q)-n\Gamma - c_{\circ} , q \rangle\,\right\} \\
&\quad \quad \quad \quad= \langle - \mathcal{L}_{\upalpha_{\circ}}u_{\circ} + f(\cdot\,,\upalpha_{\circ}(\cdot),q_{\circ})-n\Gamma - c_{\circ} , q_{\circ}\rangle.
\end{aligned}
\end{equation*}
In particular, $q_{\circ}$ is again a minimizer as it is for the problem \eqref{eq: min problem measure - proof} but where we subtract to $f$ a constant, i.e.
\begin{equation*}
    q_{\circ} \in \argmin\limits_{q\in Q}\,\big\{\;\langle -\mathcal{L}_{\upalpha_{\circ}}u_{\circ} + f(\cdot\,,\upalpha_{\circ}(\cdot), q)-n\Gamma - c_{\circ} , q \rangle\;\big\}.
\end{equation*}
The latter can be written in the notations of \textit{Step 2.1}
\begin{equation}\label{eq: min problem measure - proof 3}
    \min\limits_{q}\;\big\{\; \langle\; \widetilde{f}(\cdot\,,q)-n\Gamma + g(\cdot) \,,\, q\;\rangle\,,\quad \text{s.t.: }\; q\in \mathcal{M}^{+}_{d}(\mathds{R}^{m})\;\big\}.
\end{equation}
The first-order necessary optimality conditions \eqref{eq: nec cond measure - proof} written for the latter problem \eqref{eq: min problem measure - proof 3} now yields
\begin{equation*}
        \langle \widetilde{f}(\cdot\,,q_{\circ})-n\Gamma + g(\cdot)\,,\, h \rangle + \langle D_{\mu}\widetilde{f}(\cdot\,,\,q_{\circ})[h], q_{\circ} \rangle  \geq 0,\quad \forall\, h\in\{\,h\in\mathcal{M}_{d}(\mathds{R}^{m})\,:\, h^{-}\ll q_{\circ}\}
\end{equation*}
Thanks to assumption (C2), the second term in the above inequality is non-positive when $h$ is non-negative. So it suffices to choose $h$ as a positive measure supported on the Borel subset $B$ that we have fixed in our hypothesis, and recalling the definition of $\Gamma$, one has $\widetilde{f}(\cdot\,,q_{\circ}) + g(\cdot)-n\Gamma<0 $ for $n$ sufficiently large ($n>1$ is indeed enough) which yields a contradiction. Hence there cannot be any Borel subset of non-zero measure in which $\widetilde{f}(\cdot\,,q_{\circ}) + g(\cdot)$ is positive, i.e. $\widetilde{f}(x,q_{\circ}) + g(x) \leq 0$ $q_{\circ}$-almost everywhere in $\mathds{R}^{m}$, and together with the conclusion of \textit{Step 2.2} we finally have $\widetilde{f}(x,q_{\circ}) + g(x) = 0$ $q_{\circ}$-almost everywhere in $\mathds{R}^{m}$. We conclude with (\ref{thm regularity meas}) in Theorem \ref{thm summary diff} which ensures that $q_{\circ}$ is absolutely continuous with respect to Lebesgue measure, and hence the desired result:
\begin{equation}\label{eq: hjb - proof}
    -\mathcal{L}_{\upalpha_{\circ}}u_{\circ}(x) + f(x,\upalpha_{\circ}(x),q_{\circ}) = c_{\circ},\quad \text{almost everywhere in } \mathds{R}^{m}
\end{equation}
that is, thanks to \eqref{eq: charact upalpha - proof},  $H(x,\nabla u_{\circ}(x),D^{2}u_{\circ}(x),q_{\circ})=c_{\circ}$ a.e. in $\mathds{R}^{m}$.

\textit{Step 2.4. (Conclusion)}\\
At this stage of the proof, we have shown that $(q_{\circ},\upalpha_{\circ})$ is an optimal solution of \eqref{eq: primal - min} if and only if there exists a pair $(c_{\circ},u_{\circ})\in\mathds{R}^{m}\times \mathcal{X}$ satisfying the optimality conditions \eqref{optimality conditions - main}. And the latter conditions yield the no-\textit{duality gap}, also the growth condition of the function $u_{\circ}$ is given by the definition of $\mathcal{X}$ as in \eqref{eq: functional space X - main} (i.e. the statement (II-b)), the properties of the measure $q_{\circ}$ are ensured by (\ref{thm regularity meas}) in Theorem \ref{thm summary diff} (i.e. the statement (II-c)) and we have the characterization \eqref{eq: charact upalpha - proof} of $\upalpha_{\circ}$ (i.e. the statement (II-d)) noting that $q_{\circ}$ is equivalent to Lebesgue measure. Finally, the equation \eqref{eq: hjb - proof} together with $q_{\circ}\in \text{Ker}(\mathcal{L}_{\upalpha_{\circ}}^{*})$ and \eqref{eq: charact upalpha - proof} yield the PDE system \eqref{eq: pde system - thm}, and $(u_{\circ},q_{\circ})$ being in $W^{r,2}_{\text{loc}}(\mathds{R}^{m})\times W^{s,1}_{\text{loc}}(\mathds{R}^{m})$, for  $r>\frac{m}{2}$ and $s>m$, is a direct consequence of (\ref{thm regularity meas}) in Theorem \ref{thm summary diff} and of Theorem \ref{thm sobolev domain extension} (see Remark \ref{rmk: domain of L in sobolev}). Substituting \eqref{eq: hjb - proof} in the equation \eqref{eq: no duality gap - proof} yields the characterization of the constant $c_{\circ} = \langle f(\cdot\,, \upalpha_{\circ}(\cdot),q_{\circ}), q_{\circ} \rangle$, hence the statement (II-a). 
We are therefore left with the proof of the last statement.

\textit{Step 3. (Uniqueness of $u_{\circ}$)}\\
Assume here the primal problem (statement (I) of the theorem) enjoys uniqueness.\\ 
To prove that $u_{\circ}(\cdot)$ is unique, we need to assume in addition that the vector field $b(x,\alpha)$ is locally Lipschitz continuous with at most a linear growth in $x$, uniformly in $\alpha$, i.e. $\theta = 1$ in (A6) and hence $\kappa = d$. We also need $r>\frac{m}{2}$ in order to ensure continuity of $u_{\circ}(\cdot)$ (see \cite[Remark 5]{kouhkouh1}). This setting will allow us to apply the Liouville type result in \cite{bardi2016liouville}.

Suppose $(c_{\circ},u_{1}(\cdot)), (c_{\circ},u_{2}(\cdot))$ are two solutions with a polynomial growth of order at most $d$. Then we have, using the inequality ``$\min(A-B)\leq \min(A) - \min(B)$"
\begin{equation*}
    \min\limits_{\alpha\in A}\{\,-\mathcal{L}_{\alpha}(u_{1}-u_{2})\,\} \leq \min\limits_{\alpha\in A}\{-\mathcal{L}_{\alpha}u_{1} + f(\cdot\,,\alpha,q_{\circ})\} - \min\limits_{\alpha\in A}\{-\mathcal{L}_{\alpha}u_{2} + f(\cdot\,,\alpha,q_{\circ})\}  = 0
\end{equation*}
Therefore uniqueness of a solution $(c_{\circ},u_{\circ}(\cdot))$ is reduced to proving that there cannot exist non-constant sub-solutions to the static HJB equation $\min\limits_{\alpha\in A}\{-\mathcal{L}_{\alpha}v \} = 0$, where $v\coloneqq u_{1}-u_{2}$ i.e. whether Liouville property holds for the latter static HJB equation. This is answered positively in \cite{bardi2016liouville} using the following\\
\textit{claim: there exist a function $\psi\in C^{\infty}(\mathds{R}^m)$ and $R_{o}>0$ such that }
\begin{gather*}
	\min\limits_{\alpha\in A}\{-\mathcal{L}_{\alpha}\psi(x)\} \geq 0 \quad \text{in }\, \overline{B(0,R_{o})}^{C},\quad \psi(x) \to +\infty\; \text{when }\, |x|\to +\infty\\
	\text{ and }\quad \lim\limits_{|x|\to +\infty}\,\frac{v(x)}{\psi(x)} = 0
\end{gather*}
Hence, a Liouville type result \cite[Theorem 2.1]{bardi2016liouville} ensures that $v=u_{1}-u_{2}\equiv \text{constant}$. To prove the claim, we check that $\psi(x) \coloneqq |x|^d \log(|x|)$ works. This is done in \cite{kouhkouhPhD}.
\end{proof}

\subsection{Some remarks and examples}\label{sec: rmk and ex}

We discuss in the following remark the use of the TV-norm as it is uncustomary in the mean-field games literature.

\begin{remark}\label{rmk: TV}
The Total-Variation norm --although it is somehow dictated by the results in \S \ref{sec: duality} and \S\ref{sec: opt space measure} since $(\mathcal{M}(\mathds{R}^{m}),\|\cdot\|_{TV})$ is a Banach space (see e.g. \cite[\S IV.2.16]{dunford1988linear})-- seems to be natural in regards to our primal problem \eqref{eq: primal - min} where the constraint $q\in \text{Ker}(\mathcal{L}_{\upalpha}^{*})$ is nothing but \eqref{equation mu_diff op} in \S \ref{sec: diffusion}, that is requiring $q$ to be an invariant (stationary) measure. Therefore, there is no idea of ``transportation" which the Wasserstein metric seems to capture the best. Roughly speaking, in optimal transport, one seeks a transport plan (unknown) such that for a given initial measure, its image with the transport plan matches a given target measure. Whereas in our case, one seeks measures that remain invariant (in the sense \eqref{eq: semigroup invariance}) w.r.t. to a given analogue of the transport plan (known), that is, the $C_{0}$-semigroup $(T_{t})_{t\geq 0}$ on $L^{1}(\mathds{R}^{m},\mu)$ which has $\mathcal{L}_{\upalpha}^{*}$ as a generator. And the latter invariance needs to hold for every $t\geq 0$. Hence, one needs a stronger distance than Wasserstein and TV seems to be well suited. 
\end{remark}



\begin{remark}
A heuristic interpretation of the ergodic MFG system \eqref{eq: mfg - main} is the following: an agent aims at maximizing the payoff 
\begin{equation*}
    \limsup\limits_{T\to +\infty} \; \mathds{E}\left[ \frac{1}{T}\int_{0}^{T} f(X_{t},\alpha_{t},m(t))\,\text{d}t\right]
\end{equation*}
while controlling the trajectory \eqref{eq: SDE intro} that is
\begin{equation*}
    \text{d}X_{t} = b(X_{t},\alpha_{t})\text{d}t + \sqrt{2}\varrho(X_{t},\alpha_{t})\text{d}B_{t}
\end{equation*}
and where $m(\cdot)$ denotes the distribution of all the other agents who behave analogously. An equilibrium is reached when the distribution $m(\cdot)$ of the agents solves the FPK equation in \eqref{eq: mfg - main} for which the ergodic constant is the optimal payoff.
\end{remark}

A tentative game-theoretical interpretation of assumption (C2) in the line of the above remark is the object of the following. 
\begin{remark}\label{rmk: interpretation}
Using 
\begin{equation*}
    f(x,\alpha,\mu+h) = f(x,\alpha,\mu) + D_{\mu}f(x,\alpha,\mu)[h] + o(\|h\|),\quad \forall\,h\in\mathcal{M}_{d}(\mathds{R}^{m}),
\end{equation*}
assumption (C2) that we recall here 
\begin{equation}
    \tag{C2}
        \langle\, D_{\mu}f(\cdot\,,\alpha,\mu)[h]\,,\,\mu\, \rangle \leq 0, \quad \forall\,h,\mu\in\mathcal{M}^{+}_{d}(\mathds{R}^{m}),
\end{equation}
means
\begin{equation*}
    \lim\limits_{t\downarrow 0} \frac{1}{t}\left(\int f(x,\alpha,\mu+th)\text{d}\mu - \int f(x,\alpha,\mu)\text{d}\mu\right)
    \leq 0,
    \quad \forall\, \mu,h\in \mathcal{M}_{d}^{+}(\mathds{R}^{m}).  
\end{equation*}
The measure $h$ being in $\mathcal{M}_{d}^{+}(\mathds{R}^{m})$, the latter would mean that any positive variation in the distribution of the agents decreases the payoff $f$ in expectation.
\end{remark}

In the next remark we compare (C2) with Lasry-Lions monotonicity assumption \eqref{monotonicity LL 2} that we recall is
\begin{equation}
    \tag{M'}
    \langle\,D_{\mu}f(\cdot\,,\alpha,\mu)[h]\,,\,h\,\rangle \leq 0,\quad \forall\,\mu,h\in\mathcal{M}_{d}(\mathds{R}^{m}).
\end{equation}
\begin{remark}\label{rmk: assumption LL} 
There is a twofold difference between \eqref{monotonicity LL 2} and our assumption (C2): \\
\textbullet \, firstly, the choice of measures in \eqref{monotonicity LL 2} is the whole space $\mathcal{M}_{d}(\mathds{R}^{m})$, whereas in our case we require the assumption to hold only in the positive cone\footnote{In (C2), we ask  $\langle\,D_{\mu}f(\cdot\,,\alpha,\mu)[h]\,,\mu\,\rangle \leq 0$ to hold $\forall\,h,\mu\in \mathcal{M}_{d}^{+}(\mathds{R}^{m})$ s.t. $\mu\ll dx$. But in this ongoing discussion, we forget deliberately about absolute continuity of $\mu$ w.r.t. Lebesgue measure in order to focus rather on the \textit{structure} of the assumption when compared to \eqref{monotonicity LL 2}.} $\mathcal{M}_{d}^{+}(\mathds{R}^{m})$; \\
\textbullet\, secondly, the averaging $\langle\,\cdot\,,\,h\,\rangle$ in \eqref{monotonicity LL 2} is taken with respect to the same measure $h$ as in the Fréchet derivative $D_{\mu}f(x,\alpha,\mu)[h]$, whereas in our case, the averaging $\langle\,\cdot\,,\,\mu\,\rangle$ is taking w.r.t. the measure $\mu$ where the derivative has been computed. 

This difference makes it difficult to compare the two conditions. However, in the case $f$ depends linearly on the measure $\mu$, e.g. $f(x,\alpha,\mu) = \int_{\mathds{R}^{m}}K(x-y,\alpha)\,\text{d}\mu(y)$, then the Fréchet derivative $D_{\mu}f(x,\alpha,\mu)[h] = f(x,\alpha,h)$ is independent of $\mu$. Hence, our condition (C2) requires $\langle f(\cdot\,,\alpha,h)\,,\, \mu \rangle \leq 0$ for all $h,\mu\in\mathcal{M}^{+}_{d}(\mathds{R}^{m})$, while condition \eqref{monotonicity LL 2} writes as $\langle f(\cdot\,,\alpha,h)\,,\, h \rangle \leq 0$ for all $h\in\mathcal{M}_{d}(\mathds{R}^{m})$. Therefore, in the case of a linear dependency on the measure, (C2) is stronger than \eqref{monotonicity LL 2} when restricted to the positive cone $\mathcal{M}^{+}_{d}(\mathds{R}^{m})$. If we assume in addition that the kernel $K(\cdot\,,\alpha)$ is odd, then by direct computations using the Jordan decomposition of $h$ (see \cite{kouhkouhPhD}), one can see that condition (C2) implies \eqref{monotonicity LL 2}. 

A nonlinear version of this example can be
\begin{equation*}
    f(x,\alpha,\mu) = F\left(x,\alpha,\int_{\mathds{R}^{m}}K(x-y,\alpha)\,\text{d}\mu(y)\right).
\end{equation*}
Denote by $D_{3}$ the derivative in the third variable of $F:\mathds{R}^{m}\times A \times \mathds{R}\to \mathds{R}$, then
\begin{equation*}
\begin{aligned}
    D_{\mu}f(x,\alpha,\mu)[h] & = \int_{\mathds{R}^{m}}D_{3}F\left(x,\alpha,\int_{\mathds{R}^{m}}K(x-z,\alpha)\,\text{d}\mu(z)\right)K(x-y,\alpha)\,\text{d}h(y)\\
    & = \int_{\mathds{R}^{m}} \phi(x,\alpha,\mu)K(x-y,\alpha)\,\text{d}h(y)
\end{aligned}
\end{equation*}
where $\phi$ is the term coming from $D_{3}F$ in the previous line. In this case, \eqref{monotonicity LL 2} writes
\begin{equation*}
    \iint_{\mathds{R}^{2m}}\left[ \phi(x,\alpha,\mu)K(x-y,\alpha)\right]\;\text{d}h(y)\,\text{d}h(x) \leq 0, \quad \forall\,h\in \mathcal{M}_{d}(\mathds{R}^{m}),
\end{equation*}
and assumption (C2) is now:
\begin{equation*}
    \iint_{\mathds{R}^{m}}\left[ \phi(x,\alpha,\mu)K(x-y,\alpha)\right]\;\text{d}h(y)\,\text{d}\mu(x) \leq 0, \quad \forall\,h,\mu\in \mathcal{M}_{d}^{+}(\mathds{R}^{m}).
\end{equation*}
In this example, it is sufficient to have the term between brackets non-positive almost everywhere to satisfy assumption (C2) since $h,\mu$ are non-negative measures. But this is not sufficient to guarantee assumption \eqref{monotonicity LL 2} since $h$ can be any (signed) measure. 
\end{remark}

\textbf{Examples. } \\
Functions $f$ satisfying (A5), (B1), (C1) and (C2) are for example 
$$
f(x,\alpha,\mu) = g(x,\alpha) + F\big(x,\alpha,k(\cdot\,,\alpha)\ast \mu(x)\big)
$$
where $g(\cdot\,,\cdot),k(\cdot\,,\cdot), F(\cdot\,,\cdot\,,\mu):\mathds{R}^{m}\times A \to \mathds{R}$ satisfy (A5) and \eqref{assumption phi 1}-\eqref{assumption phi 2}, $k(\cdot\,,\alpha)\ast \mu(x) = \int_{\mathds{R}^m} k(x-y,\alpha)\,\text{d}\mu(y)$ and $F$ is either one of the following cases
\begin{enumerate}[label = (\arabic*)]
    \item $F=0$ corresponds to the setting of \cite{kouhkouh1}.

    \hfill
    
    \item If $F(x,\alpha,\mu) = k(\cdot\,,\alpha)\ast \mu(x)$, then $D_{\mu}f(x,\alpha,\mu)[h] = k(\cdot\,,\alpha)\ast h(x)$. In this case, 
    it is sufficient to have $k(\cdot\,,\cdot\,)$ bounded for all $x,\alpha$ in order to satisfy (C1), and $k(\cdot\,,\cdot)\leq 0$ in order to satisfy (C2). Indeed we have $\langle D_{\mu}f(\cdot\,,\alpha,\mu)[h],\mu \rangle = \iint k(x-y,\alpha)\,\text{d}h(y)\,\text{d}\mu(x)$ where  $h,\mu\in \mathcal{M}_{d}^{+}(\mathds{R}^{m})$.
    
    \hfill
    
    \item If $F(x,\alpha,\mu) = \frac{1}{\gamma} (k(\cdot\,,\alpha)\ast \mu(x))^{\gamma}$, then \\
    $D_{\mu}f(x,\alpha,\mu) = \big(k(\cdot\,, \alpha)\ast \mu(x)\big)^{\gamma-1} k(\cdot\,,\alpha)\ast h(x)$. So if for some constant $M>0$ we have $|k(x,\alpha)| \leq M$ for all $x,\alpha$, then 
    \begin{equation*}
        \begin{aligned}
            |D_{\mu}f(x,\alpha,\mu)[h]| & \leq \big| k(x, \alpha)\ast \mu(x)\big|^{\gamma-1} \, |k(x,\alpha)\ast h(x)|\\
            & \leq M^{\gamma-1} \|\mu\|^{\gamma-1}_{TV}\, M \|h\|_{TV} = M^{\gamma}\|\mu\|^{\gamma-1}_{TV} \|h\|_{TV}\\
            & \leq \tilde{M}\, \|h\|_{TV},\quad \text{where } \, \tilde{M}:=M^{\gamma}\|\mu\|^{\gamma-1}_{TV},
        \end{aligned}
    \end{equation*}
    which ensures (C1). Moreover we have
    \begin{equation*}
        \langle D_{\mu}f(x,\alpha,\mu)[h] , \mu \rangle = \int \left(\int k(x-z,\alpha)\text{d}\mu(z)\right)^{\gamma-1} \left(\int k(x-y,\alpha)\text{d}h(y)\right) \, \text{d}\mu(x).
    \end{equation*}
    So if $\gamma$ is odd and $k(\cdot\,,\cdot)\leq 0$, then (C2) is satisfied. \\
    A more general sufficient condition would be to have 
    \begin{equation*}
        \left(\int k(x-z,\alpha)\text{d}\mu(z)\right)^{\gamma-1} k(x-y,\alpha) \leq 0,\, \forall \, x,y,\alpha, \text{ and for } \mu\in \mathcal{M}_{d}^{+}(\mathds{R}^{m}).
    \end{equation*}
    
    \hfill
    
    \item In general, we have $D_{\mu}f(x,\alpha,\mu)[h] = \int_{\mathds{R}^{m}} \phi(x,\alpha,\mu)k(x-y,\alpha)\,\text{d}h(y)$ where $\phi(x,\alpha,\mu) = D_{3}F(x,\alpha,k(\cdot\,,\alpha)\ast \mu(x))$ and $D_{3}$ is the derivative in the third variable of $F$. \\
    So it is sufficient to have $ \phi(x,\alpha,\mu)k(x-y,\alpha)$ bounded and non-positive to satisfy (C1) and (C2).
\end{enumerate}
We refer to \cite[\S 1.3]{bardi2021convergence} and references therein for various examples of the kernel $k$ with different interpretations.

\section*{Acknowledgments}
I wish to thank Martino Bardi, J. Frédéric Bonnans and Alessandro Goffi for fruitful discussions on the content of this manuscript. I am also grateful to Sergei Zuyev for helpful discussion on \S \ref{sec: opt space measure}.

\bibliographystyle{siam}
\bibliography{references}
\end{document}